\documentclass[12pt]{amsart} 
\usepackage{amsmath, amssymb, amsthm, bm, mathtools, enumitem, caption}
\usepackage[pagebackref]{hyperref}
\usepackage{float}
\usepackage{xcolor}
\definecolor{codegray}{gray}{0.95}

\usepackage[noabbrev,capitalize]{cleveref}
\usepackage{adjustbox}
\crefname{equation}{}{}
\usepackage{ytableau}
\usepackage{graphicx, tikz}
\usepackage[textheight=9in, textwidth=6.95in]{geometry}

\usepackage{microtype}
\usepackage{bbm}
\usepackage{stmaryrd}

\newtheorem{theorem}{Theorem}[subsection]
\newtheorem{lemma}[theorem]{Lemma}
\newtheorem{assumption}[theorem]{Assumption}
\newtheorem{notation}[theorem]{Notation}

\newtheorem*{conjecture*}{Conjecture}

\theoremstyle{definition}
\newtheorem{definition}[theorem]{Definition}
\newtheorem{example}[theorem]{Example}

\theoremstyle{remark}

\numberwithin{equation}{subsection}

\DeclarePairedDelimiter\abs{\lvert}{\rvert}

\newcommand{\N}{\mathbb N}

\usepackage{bm} 

\newcommand{\qbin}[2]{\genfrac{[}{]}{0pt}{}{#1}{#2}_{q}}

\newcommand{\qbinom}[3][q]{\genfrac{[}{]}{0pt}{}{#2}{#3}_{#1}}

\newcommand{\ps}[1]{\llbracket #1 \rrbracket} 

\newcommand{\C}{\mathbb C}

\newcommand{\Q}{\mathbb Q}

\newcommand{\Z}{\mathbb Z}

\renewcommand{\abs}[1]{\left\vert #1 \right \vert}

\usepackage{color}
\definecolor{keywordcolor}{rgb}{0.7, 0.1, 0.1}   
\definecolor{tacticcolor}{rgb}{0.0, 0.1, 0.6}    
\definecolor{commentcolor}{rgb}{0.4, 0.4, 0.4}   
\definecolor{symbolcolor}{rgb}{0.0, 0.1, 0.6}    
\definecolor{sortcolor}{rgb}{0.1, 0.5, 0.1}      
\definecolor{attributecolor}{rgb}{0.7, 0.1, 0.1} 
\definecolor{backcolour}{rgb}{0.95,0.95,0.92}
\definecolor{filepathbg}{rgb}{0.88,0.88,0.84}
\definecolor{numcolor}{rgb}{0.6,0.35,0.0}

\usepackage{fvextra}
\usepackage[most]{tcolorbox}

\usepackage{iftex}
\ifLuaTeX
  \usepackage{fontspec}
\fi
\ifXeTeX
  \usepackage{fontspec}
\fi

\newcommand{\nhds}{{\char"1D4DD}}
\newcommand{\kw}[1]{\textcolor{keywordcolor}{\textbf{#1}}}
\newcommand{\var}[1]{\textcolor{tacticcolor}{#1}}
\newcommand{\prop}[1]{\textcolor{sortcolor}{#1}}
\newcommand{\num}[1]{\textcolor{numcolor}{#1}}
\newcommand{\imp}[1]{\textcolor{commentcolor}{\textbraceleft#1\textbraceright}}

\NewDocumentEnvironment{code}{m}{%
  \VerbatimEnvironment
  \begin{tcolorbox}[enhanced,breakable,sharp corners,boxrule=0pt,
    colback=backcolour,colbacktitle=filepathbg,coltitle=black,
    fonttitle=\ttfamily\small,title={\detokenize{#1}},
    left=6pt,right=6pt,top=4pt,bottom=4pt,toptitle=2pt,bottomtitle=2pt]%
  \begin{Verbatim}[commandchars=\\\{\},fontsize=\small]%
}{%
  \end{Verbatim}\end{tcolorbox}%
}

\title[The Rogers--Ramanujan identities]{
Formalized $q$-series: The Rogers--Ramanujan Identities and Beyond}


\author{Kenny Lau, Seewoo Lee, and Ken Ono}

\address{Axiom Math, 124 University Avenue, Palo Alto, CA 94301}
\email{kenny@axiommath.ai}
\email{ken@axiommath.ai}

\address{Department of Mathematics, University of California—Berkeley, Berkeley, CA 94720}
\email{seewoo5@berkeley.edu}

\begin{document}

\begin{abstract} The theory of $q$-series and basic hypergeometric series plays a crucial role at the intersection of combinatorics, number theory, and representation theory. From the classical partition identities of Euler and Jacobi to modern developments in class field theory, vertex operator algebras, and the Monstrous Moonshine conjecture, $q$-series provide the analytic framework for a wide range of profound applications. In this paper, we discuss the formalization of this theory in the Lean proof assistant, a process that requires careful design of scalable and versatile structures to reconcile formal algebraic identities with analytic convergence properties. We address these foundational challenges by focusing on the construction of $q$-Pochhammer symbols, $q$-binomial coefficients, Bailey's Lemma and similar primitives. To demonstrate the utility of this work, we provide fully verified proofs of the Jacobi triple product identity and the celebrated Rogers--Ramanujan identities, which serve as both historical and technical benchmarks for the field. This work establishes a rigorous computational foundation for the future formalization of mock theta functions, modular forms, and the diverse algebraic structures that underpin their applications across mathematics and physics. AxiomProver was used to produce the formalizations in this paper.
\end{abstract}

\maketitle

\begin{center}
\fbox{\parbox{0.75\textwidth}{%
\begin{center}
\textbf{The formalization discussed in this paper is located at:}

\url{https://github.com/AxiomMath/RogersRamanujan}
\end{center}
}}
\end{center}

\section{Introduction and Statement of Results}

Few families of functions are as hospitable, or as surprising, as $q$-series.  At the most elementary level, they are power series in a parameter $q$; but already in Euler's work on partitions and infinite products, this simple notation became a language for encoding arithmetic, combinatorial, and analytic phenomena.  In the nineteenth century, Jacobi's theta functions and triple product identity revealed that $q$-series also govern elliptic functions and modular transformations.  In Ramanujan's notebooks, and later in the work of Rogers, Schur, Bailey, Andrews, and many others, the subject acquired a distinctive character: identities that appear as innocent generating function manipulations often conceal deep structure in partition theory, modular forms, representation theory, and statistical mechanics.  For general background on the classical theory and its combinatorial and hypergeometric aspects, see \cite{AndrewsBook,GasperRahman,SillsBook}.

This breadth is also what makes the subject a natural test case for formalized mathematics.  A $q$-series identity may be read as an equality of formal power series, as a theorem about convergent analytic functions on the unit disk, as a statement about partitions, or as the character identity of an algebraic object like a Lie algebra.  Human readers move between these interpretations with little comment; a proof assistant requires each transition to be explicit.  Formalization therefore does more than certify a final equality: it exposes the definitions, coercions, convergence hypotheses, indexing conventions, and reusable lemmas on which the equality depends.  Building on the success of formal proof in major mathematical projects and on the growth of shared libraries, this paper is a step in this direction for $q$-series.

\subsection{Some fundamental $q$-series}

Our first benchmark is Jacobi Triple Product, one of the foundational identities of the subject.  It converts a bilateral theta series into an infinite product, and thereby supplies the bridge between additive and multiplicative descriptions of theta functions.  Among its many consequences are Euler's pentagonal number theorem, product expansions for classical theta functions, and the product formulae that underlie numerous partition identities and modular-form constructions.  In this paper we formalize a proof of the following standard form of the identity, originally rooted in Jacobi's theory of elliptic functions (for example, see \cite{Andrews1974, Andrews1975, Andrews1986, AndrewsBook, Jacobi1829}).

\begin{theorem}[Jacobi Triple Product]\label{thm:JTP} If we let  $q,z\in\C$ with $\abs{q}<1$ and $z\ne 0$, then we have
$$
\sum_{n\in\Z}(-1)^n z^n q^{n(n-1)/2}
=
\prod_{m=1}^{\infty}(1-q^m)(1-zq^{m-1})(1-z^{-1}q^m).
$$
\end{theorem}

This theorem plays a central role in framing Jacobi forms and half-integral weight modular forms (for example, see \cite{BFOR,EichlerZagier}). Moreover, this identity is quite versatile, and easily leads to identities such as Euler's Pentagonal Number Theorem (Theorem \ref{thm:PNT})
$$
\prod_{n=1}^{\infty} (1-q^n)=\sum_{k\in \Z} (-1)^k q^{\frac{3k^2+k}{2}},
$$
and Jacobi's identity (Theorem \ref{thm:jacobi-triangular})
$$
\prod_{n=1}^{\infty} (1-q^n)^3 = \sum_{k=0}^{\infty} (-1)^k (2k+1)q^{\frac{k^2+k}{2}}.
$$
The analytic statement in Theorem~\ref{thm:JTP} is the familiar mathematical formulation. The formal development separates the algebraic skeleton of the identity from the analytic hypotheses that justify the infinite sums and products.  This separation is one of the central design themes of this paper.

We next turn to the iconic $q$-series of Rogers and Ramanujan.  These formulas first emerged from partition theory and now reverberate across combinatorics, number theory, representation theory, and physics. To this end, we begin by recalling the fundamental $q$-Pochhammer symbols
\begin{equation}
(a;q)_n:=\prod_{j=0}^{n-1}(1-aq^{j}) \qquad\text{and}\qquad
(a;q)_\infty:=\prod_{j=0}^{\infty}(1-aq^{j}).
\end{equation}
The two Rogers--Ramanujan $q$-series are
\begin{equation}\label{RR1}
G(q):=\sum_{n\ge 0}\frac{q^{n^2}}{(q;q)_n} \qquad {\text {and}} \qquad
H(q):=\sum_{n\ge 0}\frac{q^{n(n+1)}}{(q;q)_n}.
\end{equation}
They satisfy the striking sum-product identities discovered by Rogers, rediscovered by Ramanujan, and given important finite forms by Schur \cite{Rogers1894,RR1919,Schur1917}.\footnote{Schur proved finite polynomial versions, from which the infinite product identities follow by taking the limit.}

\begin{theorem}[The Rogers--Ramanujan identities]\label{thm:RR_Identities}
As formal power series, we have
\begin{equation}
G(q)=\frac{1}{(q;q^5)_\infty\,(q^{4};q^{5})_\infty} \qquad {\text {and}} \qquad
H(q)=\frac{1}{(q^{2};q^{5})_\infty\,(q^{3};q^{5})_\infty}.
\end{equation}
\end{theorem}

The Rogers--Ramanujan identities, and their proofs, are especially well-suited as a benchmark for a formal library.  They require finite and infinite products, nontrivial reindexing of sums, control of convergence or formal limiting processes, and a precise interface between basic hypergeometric transformations and product expansions.  They also connect to a large mathematical ecosystem: partition theory, solvable lattice models, conformal field theory, vertex operator algebras, and generalized Rogers--Ramanujan identities all use variants of these ideas \cite{AndrewsBook,BaxterBook,DMS,FLM,GasperRahman,KacIDA,LepowskyLi,GOW,SillsBook}.  Moreover, after multiplying by suitable fractional powers of $q=e^{2\pi i\tau}$, with $\tau\in\mathbb{H}$, the product sides become the level $5$ modular functions
\begin{equation}\label{RRnormalized}
f_1(\tau):=q^{-1/60}G(q)\qquad {\text {and}}\qquad  f_2(\tau):=q^{11/60}H(q),
\end{equation}
placing the identities directly in the web of modular forms and modular functions (see these references for background on modular forms and functions \cite{DiamondShurman,OnoCBMS,Schoeneberg,ZagierRR}).

The paper is organized as follows.  In \Cref{sec:Formalizing_q_series} we describe the mathematical architecture of the Lean development: the treatment of finite and infinite products, the definitions of $q$-Pochhammer symbols and $q$-binomial coefficients, and the formalization of Bailey pairs and Bailey's Lemma.  In \Cref{sec:JTP} we explain the formalized proof of \Cref{thm:JTP}, emphasizing the translation between bilateral sums, products, and the theorem statement used by later sections.  In \Cref{sec:RR} we prove \Cref{thm:RR_Identities} inside the library and describe how the proof tests the scalability of the preceding infrastructure.  Finally, \Cref{sec:Appendix} records the role of the AxiomProver system in the project and gives reproducibility information for the formal artifact. AxiomProver is an AI system for mathematical research via formal proof that is currently under development by Axiom Math.

\subsection*{Acknowledgements}

AI-assisted tools were used in the formalization work underlying this paper. Most notably, AxiomProver, an AI theorem-proving system currently under development by Axiom Math, was used to formalize several of the results presented here. A detailed account of the formalizations carried out using AxiomProver is provided in Appendix \ref{sec:axiomprover}. Claude Code was also used in a supporting role, primarily for sketching formalization approaches and assisting with subsequent cleanup.

\section{A prelude to formalization}

The purpose of this section is to explain mathematical background and design principles for the library. The central formalization problem is that classical notation suppresses several distinctions: a finite product may later be sent to a limit, a formal product may be interpreted analytically, and the same symbol $(a;q)_n$ is used in rings, fields, normed algebras, and spaces of formal series.  Our development is designed to make these transitions explicit while keeping the user-facing statements close to ordinary mathematical language.  We work in Lean \cite{Lean}, building on \texttt{mathlib} \cite{Mathlib2020}, and isolate the reusable primitives needed for Jacobi Triple Product, Bailey's Lemma \cite{Bailey1949}, and the Rogers--Ramanujan identities.

\subsection{Mathematical Background} In this subsection we develop the necessary mathematical prerequisites we require.

\subsubsection{Assumptions}

All rings in this paper are assumed to be commutative. In accordance with Lean, $\N$ (the set of natural numbers) contains $0$. When necessary, we denote $\N \setminus \{0\}$ by $\N^+$, also in accordance with Lean.

\subsubsection{Typeclass resolution}

Lean has a built-in typeclass resolution system \cite{SelsamTypeclass} which is very powerful. For example, a field is automatically a ring. Formally, if we assume \texttt{(F : Type*) [Field F]}, then Lean can automatically populate the type \texttt{Ring F} (with the mathematically correct term). This means that definitions and theorems only have to be stated with the most general typeclass assumptions, and then they will be able to be applied in a wide variety of situations. To this end, $(a; q)_n$ is defined over commutative rings, and $(a; q)_\infty$ is defined over any commutative ring that also has a topological space structure, without any requirement of compatibility between the ring and the topology. Namely, in Lean syntax we have the following definitions.

\begin{code}{RogersRamanujan/NumberTheory/QTheory/Defs.lean}
\kw{def} qPochhammer \imp{R} [CommRing \var{R}] (\var{a} \var{q} : \var{R}) (\var{n} : ℕ) : \var{R} :=
  ∏ \var{i} ∈ range \var{n}, (\num{1} - \var{a} * \var{q} ^ \var{i})
\end{code}

\begin{code}{RogersRamanujan/NumberTheory/QTheory/Defs.lean}
\kw{def} qPochhammerInf \imp{R} [TopologicalSpace \var{R}] [CommRing \var{R}] (\var{a} \var{q} : \var{R}) : \var{R} :=
  ∏'[conditional ℕ] \var{i}, (\num{1} - \var{a} * \var{q} ^ \var{i})
\end{code}

\subsubsection{Analytic vs. Algebraic}

Theorems in $q$-series are often stated analytically, in terms of $q \in \C$, and where power series represent analytic functions on the unit complex disc. However, those objects are difficult to work with algebraically, for example if we consider the ring of power series, then asking about convergence puts an extra level of complication. For example, the set of convergent (with a fixed radius) power series in $q$ is not closed under the natural $q$-adic topology. Therefore, we have decided to completely transform the theory into an algebraic one, and replace the analytic objects with non-archimedean rings.

As it turns out, even the assumption of being non-archimedean is still too weak if we want $(a; q)_\infty$ to actually converge, since the non-archimedean axiom does not put any constraint on the multiplicative structure. We know that $q$ should be topologically nilpotent, so that the multiplicands $1 - aq^n$ tend to $1$. However, this does not guarantee convergence, even for non-archimedean rings (see the counterexample in Theorem \ref{thm:no-poch-inf}). To this end, we define a new class of rings which we call  \texttt{StrongNonarchimedeanRing} in the project. Recall that a non-archimedean ring is a topological ring where the set of open additive subgroups forms a basis of the neighborhoods of $0$. A \textbf{strongly non-archimedean ring} is a topological ring where the set of \textit{open additive subrings without unity} forms a basis of the neighborhoods of $0$. In Lean:

\begin{code}{RogersRamanujan/Topology/Algebra/Nonarchimedean/Strong.lean}
\kw{class} StrongNonarchimedeanRing (\var{R}) [Ring \var{R}] [TopologicalSpace \var{R}] :
    Prop \kw{extends} NonarchimedeanRing \var{R} \kw{where}
  exists_mul_subset_self (\var{U} : Set \var{R}) (\var{hU} : \var{U} ∈ \nhds \num{0}) :
    ∃ \var{V} : Set \var{R}, \num{0} ∈ \var{V} ∧ IsOpen \var{V} ∧ \var{V} ⊆ \var{U} ∧ \var{V} * \var{V} ⊆ \var{V}
\end{code}

In this class of rings we recover a classical theorem: if $\sum a_n$ converges absolutely, then $\prod (1 + a_n)$ converges. This theorem is well-known for complex numbers but for non-archimedean settings, it turns out that we need a \textit{strongly} non-archimedean ring to recover the theorem. Namely, we have that if $a_n\rightarrow 0$, then $\prod(1+a_n)$ converges. In Lean:

\begin{code}{RogersRamanujan/Topology/Algebra/Nonarchimedean/Strong.lean}
\kw{theorem} multipliable_one_add_of_tendsto_zero \imp{R} [CommRing \var{R}] [UniformSpace \var{R}]
  [CompleteSpace \var{R}] [StrongNonarchimedeanRing \var{R}] [IsUniformAddGroup \var{R}]
  \imp{a : ℕ → R} (\var{ha} : Tendsto \var{a} atTop (\nhds \num{0})) : Multipliable (\num{1} + \var{a} ·) :=
\end{code}

It is tempting to resort to ``linear topology'' \footnote{\url{https://leanprover-community.github.io/mathlib4_docs/Mathlib/Topology/Algebra/LinearTopology.html\#IsLinearTopology}}, which is a common mistake. This is in fact not general enough, as the $q$-adic topology on the field of Laurent series is not a linear topology, but we would want to talk about negative powers such as $(-q^{-1}; q)_\infty$.

\subsubsection{Universal rings}

 Identities can often be abstracted to a ``universal ring'', i.e. a polynomial-like ring where that identity holds, and then every other instance of that identity in any ring can be derived from universal ring homomorphisms. For example, the identity $$x^2 - y^2 = (x-y)(x+y)$$ takes place in $\Z[x,y]$ which can then be mapped to any commutative ring with two chosen elements $a$ and $b$ to deduce $a^2 - b^2 = (a-b)(a+b)$. We can think of the one equality in the universal ring as the ``main ingredient'' of this identity.

This might apply to the Jacobi triple product identity, but the authors have not found an easy and clean approach, which complicates our approach to the Jacobi triple product identity. On the other hand, we will not face these issues for the two Rogers--Ramanujan identities. This will be explained more in Section \ref{sec:relax}.
The authors found an alternative pathway to extract the main ingredient of the Jacobi triple product identity: as a natural-indexed family of equalities in $\Z\ps{q}$, which have natural combinatorial interpretations. This will be explained more in Lemma \ref{lem:jtp-iden}.

\subsubsection{A new theory of multivariate Laurent series}

We often need to ``borrow'' negative powers to make proofs easier. For example, the statement itself of the Jacobi triple product identity mentions $a^{-1}$, and the proof is easier if we involve $q^{-1}$.

The authors failed to find satisfactory theories of multivariate Laurent series, so we have invented the following implementation as mathlib already supports Hahn series\footnote{\url{https://leanprover-community.github.io/mathlib4_docs/Mathlib/RingTheory/HahnSeries/Basic.html}}, defined for an ordered additive group $\Gamma$ over a base ring $R$. It consists of functions $\Gamma \to R$ whose support is well-founded under the order. We then set $\Gamma$ to be finitely supported functions from the set of formal variables to $\Z$, and the ordering is by the sum, which we call the total degree. The ordering is defined by the following:
\[
\left\{
\begin{matrix}
    f < g &\iff& \sum_i f_i < \sum_i g_i \\
    f \le g &\iff& (\sum_i f_i < \sum_i g_i) \lor (f = g)
\end{matrix}
\right.
\]

Then, well-founded subsets are those whose total degree is bounded below, and at each degree, there are only finitely many contributors.
Its ``ring of integers'' consists of the functions supported on exponents with non-negative total degrees, and this ring is strictly bigger than the power series (for non-trivial situations). For example, $xy^{-1}$ has total degree $0$ but is not a power series.
The authors found this theory satisfactory because the ring of integers contains the power series, coefficients can be extracted from each element, and the ring is complete under the topology induced by ``total degrees''.

\subsection{Design Principles; formal, algebraic, and analytic considerations}
Before we consider concrete formalization, we offer further comments that explicate some of the choices we make later.

\subsubsection{The philosophy of junk value}

People often talk about functions with a restricted domain. For example, the square-root ``function'' $\sqrt{x}$ is considered to be defined only for $x \ge 0$. This introduces type-theoretic baggage as from a type-theory standpoint, $x$ could \textit{a priori} be any real number, so a completely literal translation of this into Lean would mean that every time you take the square root of $x$, you have to include a proof of $x \ge 0$. This is why mathlib chose a different path \cite{Buzzard2020}: in mathlib, \texttt{Real.sqrt}\footnote{\url{https://leanprover-community.github.io/mathlib4_docs/Mathlib/Analysis/SpecialFunctions/Sqrt.html\#Real.sqrt}} is defined for all real numbers, by taking on the junk value $0$ for negative inputs. The advantage of this approach is the \textit{separation of data and proofs}, which among other things makes statements look cleaner.

In light of this, we found that in this project we often need to talk about multiplicative inverses, but we again do not want to always have to provide the proof that an element is invertible, so we have invented a function \texttt{bInv x} which returns the both-sided inverse of $x$ if it exists, and the junk value $1$ otherwise. In Lean:

\begin{code}{RogersRamanujan/Algebra/Group/Units/Basic.lean}
\kw{def} bInv \imp{α} [Monoid \var{α}] (\var{x} : \var{α}) : \var{α} :=
  \kw{open} \kw{scoped} Classical \kw{in} \kw{if} \var{h} : IsUnit \var{x} \kw{then} ↑\var{h}.\prop{unit}⁻¹ \kw{else} \num{1}
\end{code}

However, it is not always beneficial to separate data and proofs. For example, sometimes it makes proving less efficient. To this end, mathlib has \textit{both} versions when it comes to topological summation. \texttt{HasSum f a}\footnote{\url{https://leanprover-community.github.io/mathlib4_docs/Mathlib/Topology/Algebra/InfiniteSum/Defs.html\#HasSum}} means that the (suitably defined) limit of summing over the values of the function $f$ is $a$, which inherently is the proof, whereas \texttt{tsum}\footnote{\url{https://leanprover-community.github.io/mathlib4_docs/Mathlib/Topology/Algebra/InfiniteSum/Defs.html\#tsum}} is a total function with junk value so that \texttt{tsum f} returns an arbitrary limit of the above sum if it exists, and the junk value $0$ otherwise. We found that for this project both are useful to have.

\subsection{Previous formalizations of $q$-series and related topics}

As of the time of writing this paper, there is no formalization of a proper $q$-theory in \texttt{mathlib} (version v4.30.0).
However, there are some formalizations of relevant theory as separate projects.
Notably, RepoProver \cite{repoprover} formalized Grinberg's algebraic combinatorics textbook \cite{grinberg}, where Chapter 4 of the book introduces partitions, $q$-binomial coefficients, Jacobi Triple Product, and Euler's Pentagonal Number Theorem.
These are all (auto)formalized by RepoProver\footnote{\url{https://github.com/facebookresearch/algebraic-combinatorics}}, but there are some differences with our formalization.

First, their version of the Jacobi triple product identity is the following: for $a, b \in \mathbb{Z}$ with $a > 0$ and $a \ge |b|$, and $u, v \in \mathbb{Q}$ with 
\begin{equation}
\label{eqn:jtb_grinberg}
    \prod_{n = 1}^{\infty} (1 + u^{2n-1}v x^{(2n-1)a+b})(1 + u^{2n-1}v^{-1} x^{(2n-1)a - b})(1 - u^{2n}x^{2na}) = \sum_{l \in \mathbb{Z}} u^{l^2} v^{l} x^{al^2 + bl}.
\end{equation}
(See Theorem 4.3.2 of \cite{grinberg} and its proof.)
The assumptions on $a$ and $b$ are needed to ensure that all the exponents in \eqref{eqn:jtb_grinberg} are nonnegative, so that the identity holds in $\mathbb{Q}\llbracket x \rrbracket$.
Then the original Jacobi triple product identity can be recovered with $q = ux^a$ and $z = vx^b$, which holds in $\mathbb{Z}[z^{\pm}]\llbracket q \rrbracket$ (see \texttt{jacobi\_triple\_product\_fps'}).

For comparison, the formalization we offer (i.e., Theorem \ref{thm:jtp-general}) holds for more general rings (\texttt{StrongNonarchimedeanRing}).
In the case of Euler's pentagonal number theorem, RepoProver formalized it as an identity of power series over $\mathbb{Q}$, while our formalization (Theorem \ref{thm:PNT}) again holds for any \texttt{NonarchimedeanRing}.
Wang also formalized Euler's pentagonal number theorem\footnote{\url{https://github.com/wwylele/PentagonalNumberTheorem}} in Lean. One version holds for general $T_2$ topological rings, which requires summability and convergence of relevant infinite series and products as hypothesis (see \texttt{pentagonalNumberTheorem\_generic}). Another version is as an identity over a power series ring over $T_2$ topological rings, which is more general than RepoProver's formalization (see \texttt{pentagonalNumberTheorem\_intPos\_powerSeries}), but still not as general as ours.

Jacobi Triple Product, Euler's Pentagonal Number Theorem, and the Rogers--Ramanujan identities are also formalized in Isabelle \cite{isabelle} by Eberl \cite{isabelle-theta,isabelle-pentagonal,isabelle-rr}.
The proof of the Jacobi triple product identity follows Ramanujan's original argument in his lost notebook \cite{Berndt} using Ramanujan's theta function
$$
f(a,b) = \sum_{n=-\infty}^{\infty} a^{\frac{n(n+1)}{2}} b^{\frac{n(n-1)}{2}},
$$
where the final result is stated for complex numbers (see \texttt{ramanujan\_theta\_triple\_product\_complex} and \texttt{jacobi\_theta\_nome\_triple\_product\_complex}), and Euler's pentagonal number theorem follows from it.
The proof of the Rogers--Ramanujan identities follows the argument in \cite{AndrewsPartition}, whereas our work makes fundamental use of Bailey's lemma. Instead, it proves certain finite identities and takes limits of them, and applies the Jacobi triple product identity to complete the proof.

\section{Formalizing \texorpdfstring{$q$}{q}-series}\label{sec:Formalizing_q_series}

\subsection{Finite products, infinite products, and controlled limits}

In mathlib, taking a finite product of elements requires a commutative monoid. Then, an arbitrary product\footnote{\url{https://leanprover-community.github.io/mathlib4_docs/Mathlib/Topology/Algebra/InfiniteSum/Defs.html\#tprod}} require a topology, and is defined using filters and can be optionally customized using summation filters. We first state the definition in full generality, and then specialize to two cases that are often used.

\begin{definition}
    A \textbf{summation filter}\footnote{\url{https://leanprover-community.github.io/mathlib4_docs/Mathlib/Topology/Algebra/InfiniteSum/SummationFilter.html\#SummationFilter}} on $\alpha$ is a filter $L$ on $|\alpha|^{<\omega}$, the set of finite subsets of $\alpha$, often assumed to satisfy the following two conditions:
\begin{enumerate}
    \item $\varnothing \notin L$
    \item For any finite set $S \subseteq \alpha$, $\{T \in |\alpha|^{<\omega} \mid T \supseteq S\} \in L$.
\end{enumerate}
\end{definition}

\begin{definition}
    Define ``$y \in \beta$ is a product of $f : \alpha \to \beta$ under the summation filter $L$'' to mean that for any neighborhood $U$ of $y$, there is $S_0 \in L$, such that for any $S \in S_0$, the finite product $\prod_{x \in S} f(x)$ is in $U$.\footnote{\url{https://leanprover-community.github.io/mathlib4_docs/Mathlib/Topology/Algebra/InfiniteSum/Defs.html\#HasProd}}
\end{definition}

Note that such $y$ is not necessarily unique, unless $\beta$ is a $T_2$-space.

\begin{definition}
    The \textbf{unconditional} summation filter on $\alpha$ is the summation filter generated by $\{T \in |\alpha|^{<\omega} \mid T \supseteq S\}$ for $S \in |\alpha|^{<\omega}$.\footnote{\url{https://leanprover-community.github.io/mathlib4_docs/Mathlib/Topology/Algebra/InfiniteSum/SummationFilter.html\#SummationFilter.unconditional}}
\end{definition}

Note that unless specified otherwise, the unconditional summation filter is assumed when we talk about arbitrary products.

\begin{definition}
    The \textbf{conditional} summation filter on $\N$ is the summation filter generated by $\{[0,n] \mid n \ge N\}$ for $N \in \N$.\footnote{\url{https://leanprover-community.github.io/mathlib4_docs/Mathlib/Topology/Algebra/InfiniteSum/SummationFilter.html\#SummationFilter.conditional}}
\end{definition}

Note that it is defined more generally, but we often only need the one for $\N$. We now explain their names:

\begin{theorem}
    The unconditional summation filter picks out unconditional convergence, in the sense that $y \in \beta$ is unconditionally a product of $f : \alpha \to \beta$ iff for any neighborhood $U$ of $y$, there is a finite set $S \subseteq \alpha$, such that for any finite $T \supseteq S$, the finite product $\prod_{x \in T} f(x)$ is in $U$.

    The conditional summation filter picks out conditional convergence, in the sense that $f : \N \to \beta$ has product $y \in \beta$ iff $\lim_{n \to \infty} \left( \prod_{i=0}^{n-1} f(i) \right) = y$.
\end{theorem}

\subsection{Basic objects of \texorpdfstring{$q$}{q}-theory and their functorial behavior}

We formalized the following $q$-integer, $q$-factorial, $q$-binomial coefficients, and $q$-Pochhammer symbol, as building blocks.

\begin{definition}
    The $q$-integer $[n]_q := \sum_{i=0}^{n-1} q^i$, where $n \in \N$ and $q \in R$ where $R$ is an arbitrary commutative semiring. Here we make the decision that $q$ would be an element of a ring-like object, rather than a formal variable.
\end{definition}

\begin{definition}
    The $q$-factorial $[n]_q ! := \prod_{i=1}^n [i]_q$, with the same assumptions as above.
\end{definition}

\begin{definition}
    The $q$-binomial coefficients, which we have chosen to define recursively to better adhere to Lean's internal structure:
    \begin{equation}\begin{cases}
      \qbinom{n}{0} & := 1 \\
      \qbinom{0}{k+1} & := 1 \\
      \qbinom{n+1}{k+1} & := \qbinom{n}{k} + q^{k+1} \qbinom{n}{k+1}
    \end{cases}\end{equation}
    with the same assumptions as above.
\end{definition}

\begin{definition}
    Now we assume $R$ is a commutative ring, because we require subtraction. Then, the finite $q$-Pochhammer symbol is $(a; q)_n := \prod_{i=0}^{n-1} (1 - aq^i)$.
\end{definition}

\begin{definition}
    Now additionally assume $R$ has a topology. Then, we define the infinite $q$-Pochhammer symbol $(a; q)_\infty$ to be the conditional infinite product $\prod_{i=0}^{\infty} (1 - aq^i)$. Note that if it exists, then it is equal to $\lim_{n \to \infty} (a; q)_n$. In compliance with mathlib's philosophy of junk values, the value is assigned to be $1$ if the product does not exist. Note that in actual applications in this paper, we can prove unconditional convergence of the product; see Lemma \ref{lem:poch-inf-exists}.
\end{definition}

For the above definitions, we have the following in Lean:

\begin{code}{RogersRamanujan/NumberTheory/QTheory/Defs.lean}
\kw{def} qInt \imp{R} [CommSemiring \var{R}] (\var{q} : \var{R}) (\var{n} : ℕ) : \var{R} :=
  ∑ \var{i} ∈ range \var{n}, \var{q} ^ \var{i}

\kw{def} qFactorial \imp{R} [CommSemiring \var{R}] (\var{q} : \var{R}) (\var{n} : ℕ) : \var{R} :=
  ∏ \var{i} ∈ range \var{n}, qInt \var{q} (\var{i} + \num{1})

\kw{def} qChoose \imp{R} [CommSemiring \var{R}] (\var{q} : \var{R}) : ℕ → ℕ → \var{R}
  | _, \num{0} => \num{1}
  | \num{0}, _ + \num{1} => \num{0}
  | \var{n} + \num{1}, \var{k} + \num{1} => qChoose \var{q} \var{n} \var{k} + \var{q} ^ (\var{k} + \num{1}) * qChoose \var{q} \var{n} (\var{k} + \num{1})

\kw{def} qPochhammer \imp{R} [CommRing \var{R}] (\var{a} \var{q} : \var{R}) (\var{n} : ℕ) : \var{R} :=
  ∏ \var{i} ∈ range \var{n}, (\num{1} - \var{a} * \var{q} ^ \var{i})

\kw{def} qPochhammerInf \imp{R} [TopologicalSpace \var{R}] [CommRing \var{R}] (\var{a} \var{q} : \var{R}) : \var{R} :=
  ∏'[conditional ℕ] \var{i}, (\num{1} - \var{a} * \var{q} ^ \var{i})
\end{code}

\begin{notation}
    Apart from the notation $(a; q)_n$ and $(a; q)_\infty$ for the finite and infinite q-Pochhammer symbol respectively, we also write $(a)_n$ and $(a)_\infty$ respectively when $q$ is clear.
\end{notation}

We make the same convenience in the Lean code, where the former notations are available upon \texttt{open QTheory}, and the latter upon \texttt{open QTheoryUnsafe}. They are called ``unsafe'' because using the name $q$ is inherently fragile.
Every object defined here is functorial. To be precise, we have the following.

\begin{theorem}
If $S$ is a commutative (semi)ring and $f: R \to S$ is a (semi)ring homomorphism, then we have the following four identities:
\begin{equation}
f([n]_q) = [n]_{f(q)}
\end{equation}
\begin{equation}
f([n]_q!) = [n]_{f(q)}!
\end{equation}
\begin{equation}
f\left(\qbinom{n}{r}\right) = \qbinom[f(q)]{n}{r}
\end{equation}
\begin{equation}
f\left((a;q)_n\right) = (f(a); f(q))_n
\end{equation}
\end{theorem}

We delay stating the functoriality of the infinite $q$-Pochhammer product $(a; q)_\infty$ until Lemma \ref{lem:poch-inf-func}.

\subsection{An exposition of strongly non-archimedean rings}
\label{sec:strong}

Since this is a new mathematical object, we briefly describe its definition and properties.

Recall that a non-archimedean ring is a topological ring such that the set of open additive subgroups forms a basis of the neighborhood filter of $0$. In other words, if $U$ is a neighborhood of $0$, then there is an open additive subgroup $V$ such that $V \subseteq U$.

\begin{definition}
    A \textbf{strongly non-archimedean ring} is a topological ring such that for any neighborhood $U$ of $0$, there is an open subrng $V$ such that $V \subseteq U$. A subrng (subring without identity) is a subset containing $0$ and is closed under addition, negation, and multiplication, but does not necessarily contain $1$.
\end{definition}

Examples of such rings are abundant.
\begin{example}
Any normed ring satisfying the ultrametric inequality, such as $\Q_p$, $R(\!(q)\!)$, is strongly non-archimedean.
More generally, our new construction of the ring of multivariate Laurent series is also strongly non-archimedean.
\end{example}

This corresponds to the following in Lean syntax:
\begin{code}{RogersRamanujan/Topology/MetricSpace/Ultra/Basic.lean}
\kw{instance} \imp{F} [NormedRing \var{F}] [IsUltrametricDist \var{F}] : StrongNonarchimedeanRing \var{F}
\end{code}

\begin{code}{RogersRamanujan/RingTheory/MvLaurentSeries/Basic.lean}
\kw{instance} \imp{σ R} [Ring \var{R}] : StrongNonarchimedeanRing (MvLaurentSeries \var{σ} \var{R})
\end{code}

\begin{example}
Since ideals are subrngs, any topological ring $R$ with an open subring $R_0$ such that $R$ is $R_0$-linearly topologized is strongly non-archimedean. Here $R_0$-linearly topologized means that open ideals of $R_0$ form a neighborhood basis of $0$.
So in particular, all Huber rings are strongly non-archimedean.
\end{example}

This corresponds to the following in Lean syntax:
\begin{code}{RogersRamanujan/Topology/Algebra/Nonarchimedean/Strong.lean}
\kw{instance} IsLinearTopology.instStrongNonarchimedeanRing \imp{R} [Ring \var{R}] [TopologicalSpace \var{R}]
  [IsTopologicalRing \var{R}] [IsLinearTopology \var{R} \var{R}] : StrongNonarchimedeanRing \var{R}
\end{code}

\begin{example}
An arbitrary product of strongly non-archimedean rings is strongly non-archimedean. This gives us the fact that not all strongly non-archimedean rings are Huber rings. See Section \ref{sec:strong-not-open} for more details.
\end{example}

A very useful property of complete non-archimedean rings $R$ is that for any sequence $f: \N \to R$, $\sum f$ converges iff $\lim_{n \to \infty} f(n) = 0$. Here we list useful properties of \textit{strongly non-archimedean rings} $R$ that are not enjoyed by all non-archimedean rings (see Theorem \ref{thm:not-strong}):

\begin{theorem}\label{thm:strong-prop}
    Let $(a_n)_{n \in \N}$ be a sequence in $R$.

    If $\lim_{n \to \infty} a_n = 0$, then $\lim_{n \to \infty} \prod_{i=0}^{n-1} a_i = 0$.

    As a very useful special case, if $q$ is topologically nilpotent, then $\lim_{n \to \infty} q^n = 0$, and since multiplication is continuous, we get $\lim_{n \to \infty} aq^n = 0$, and so by the previous property, we have $\lim_{n \to \infty} a^n q^{\binom{n}{2}} = 0$.

    Assuming $R$ is complete, if $\lim_{n \to \infty} a_n = 0$, then $\prod_n (1 + a_n)$ converges unconditionally.
\end{theorem}

In Lean:

\begin{code}{RogersRamanujan/Topology/Algebra/Nonarchimedean/Strong.lean}
\kw{theorem} prod_tendsto_zero \imp{R} [CommRing \var{R}] [TopologicalSpace \var{R}] [StrongNonarchimedeanRing \var{R}]
  (\var{a} : ℕ → \var{R}) (\var{ha} : Tendsto \var{a} atTop (\nhds \num{0})) :
  Tendsto (∏ \var{i} ∈ range ·, \var{a} \var{i}) atTop (\nhds \num{0}) := \kw{by}
\end{code}

\begin{code}{RogersRamanujan/Topology/Algebra/Nonarchimedean/Strong.lean}
\kw{theorem} multipliable_one_add_of_tendsto_zero \imp{R} [CommRing \var{R}] [UniformSpace \var{R}]
  [CompleteSpace \var{R}] [StrongNonarchimedeanRing \var{R}] [IsUniformAddGroup \var{R}]
  \imp{a : ℕ → R} (\var{ha} : Tendsto \var{a} atTop (\nhds \num{0})) : Multipliable (\num{1} + \var{a} ·) :=
\end{code}

We then have:
\begin{lemma}[Existence of infinite $q$-Pochhammer symbol]\label{lem:poch-inf-exists}
    If $R$ is a complete strongly non-archimedean ring and $a, q \in R$ with $q$ topologically nilpotent, then $(a; q)_\infty := \prod_{n} (1 - aq^n)$ exists unconditionally.
\end{lemma}
\begin{proof}
    $\lim_{n \to \infty} q^n = 0$, so $\lim_{n \to \infty} aq^n = 0$, so by the third fact above, $\prod_{n} (1 - aq^n)$ converges unconditionally.
\end{proof}

In Lean:

\begin{code}{RogersRamanujan/NumberTheory/QTheory/StrongNonarchimedean.lean}
\kw{theorem} multipliable_one_sub_mul_pow \imp{R} [CommRing \var{R}] [UniformSpace \var{R}] [CompleteSpace \var{R}]
  [StrongNonarchimedeanRing \var{R}] [IsUniformAddGroup \var{R}]
  \imp{a q : R} (\var{hq} : IsTopologicallyNilpotent \var{q}) :
  Multipliable (\num{1} - \var{a} * \var{q} ^ ·) :=
\end{code}

\begin{lemma}[Functoriality of infinite $q$-Pochhammer symbol]\label{lem:poch-inf-func}
Let $R$ be a complete strongly non-archimedean ring, and let $S$ be a ring that is a $T_2$ space, and let $f: R \to S$ be a ring homomorphism that is continuous, and let $a, q \in R$ with $q$ topologically nilpotent. Then:
\begin{equation}
f\left( (a; q)_\infty \right) = (f(a); f(q))_\infty
\end{equation}
\end{lemma}

In Lean:

\begin{code}{RogersRamanujan/NumberTheory/QTheory/StrongNonarchimedean.lean}
\kw{theorem} map_qPochhammerInf \imp{R S F} [CommRing \var{R}] [UniformSpace \var{R}]
  [StrongNonarchimedeanRing \var{R}] [CompleteSpace \var{R}] [IsUniformAddGroup \var{R}]
  [CommRing \var{S}] [TopologicalSpace \var{S}] [T2Space \var{S}]
  [FunLike \var{F} \var{R} \var{S}] [RingHomClass \var{F} \var{R} \var{S}] (\var{f} : \var{F}) (\var{hf} : Continuous \var{f})
  (\var{a} : \var{R}) \imp{q : R} (\var{hq} : IsTopologicallyNilpotent \var{q}) :
  \var{f} (\var{a}; \var{q})_∞ = (\var{f} \var{a}; \var{f} \var{q})_∞ :=
\end{code}

\subsection{\texorpdfstring{$q$}{q}-binomial coefficients and finite \texorpdfstring{$q$}{q}-identities}

The identities in this section do not involve topology (i.e. arbitrary sum or product), so we can just assume that $R$ is a commutative ring. As a result, we can also treat each identity as taking place in the universal ring $\Z[\mathrm{vars}]$ where vars is the (finite) set of variables involved.

\begin{lemma}[$q$-Binomial theorem]
    \label{lem:binom-thm}
    \begin{equation}
        (a)_n = \sum_{k=0}^{n} \begin{bmatrix} n \\ k \end{bmatrix}_q (-a)^k q^{\frac{k(k-1)}{2}}  
    \end{equation}
\end{lemma}

In Lean:

\begin{code}{RogersRamanujan/NumberTheory/QTheory/Basic.lean}
\kw{theorem} qPochhammer_eq_sum_qChoose \imp{R} [CommRing \var{R}] \imp{a q : R} \imp{n : ℕ} :
  (\var{a}; \var{q})_\var{n} = ∑ \var{k} ∈ range (\var{n} + \num{1}), qChoose \var{q} \var{n} \var{k} * (-\var{a}) ^ \var{k} * \var{q} ^ \var{k}.\prop{choose} \num{2} :=
\end{code}

\begin{proof}
    Induction on $n$.
\end{proof}

\begin{lemma}[$q$-Cauchy identity]
    \label{lem:binomial_cauchy}
    \begin{equation}
        (uv)_n = \sum_{k=0}^{n} \begin{bmatrix}
            n \\ k
        \end{bmatrix}_q (u)_k v^k (v)_{n-k}
    \end{equation}
\end{lemma}

In Lean:

\begin{code}{RogersRamanujan/NumberTheory/QTheory/Basic.lean}
\kw{theorem} qPochhammer_mul_eq_sum_qChoose \imp{R} [CommRing \var{R}] \imp{q : R} (\var{u} \var{v} : \var{R}) (\var{n} : ℕ) :
  (\var{u} * \var{v}; \var{q})_\var{n} = ∑ \var{k} ∈ range (\var{n} + \num{1}), qChoose \var{q} \var{n} \var{k} * (\var{u}; \var{q})_\var{k} * \var{v} ^ \var{k} * (\var{v}; \var{q})_(\var{n} - \var{k}) :=
\end{code}

\begin{proof}
    Induction on $n$, using
    $$
        \begin{bmatrix} n+1 \\ k+1 \end{bmatrix}_q = q^{n-k} \begin{bmatrix}
            n \\ k
        \end{bmatrix}_q + \begin{bmatrix}
            n \\ k + 1
        \end{bmatrix}_q
    $$
    and $(a)_{n+1} = (a)_n \cdot (1 - aq^{n})$.
\end{proof}

\begin{lemma}
    \label{lem:binomial_mul}
    For $n, k, s$ with $s \le k$, we have
    \begin{equation}
        \qbin{n}{k} \qbin{k}{s} = \qbin{n}{s} \qbin{n-s}{k-s}
    \end{equation}
\end{lemma}

In Lean:

\begin{code}{RogersRamanujan/NumberTheory/QTheory/Basic.lean}
\kw{theorem} qChoose_mul \imp{R} [CommRing \var{R}] (\var{q} : \var{R}) {\var{n} \var{k} \var{s} : ℕ} (\var{hsk} : \var{s} ≤ \var{k}) :
  qChoose \var{q} \var{n} \var{k} * qChoose q \var{k} \var{s} = qChoose \var{q} \var{n} \var{s} * qChoose q (\var{n} - \var{s}) (\var{k} - \var{s}) :=
\end{code}

\begin{proof}
    By the identity
    $$
        \begin{bmatrix} n \\ k \end{bmatrix}_q = \frac{[n]_q!}{[k]_q! [n-k]_q!},
    $$
    one can easily check that both sides are equal to $\frac{[n]_q!}{[n-k]_q! [k-s]_q!}$.
\end{proof}

The following lemma will be used to prove the $q$-Pfaff--Saalsch\"utz identity (see Theorem \ref{thm:qPS-cleared}).
\begin{lemma}
    \label{lem:binomial_cauchy_2}
    \begin{equation}
        (ac)_n (bc)_n = \sum_{k=0}^{n} \begin{bmatrix} n \\ k \end{bmatrix}_q (a)_k (b)_k c^k (c)_{n-k} (abcq^k)_{n-k}
    \end{equation}
\end{lemma}

In Lean:

\begin{code}{RogersRamanujan/NumberTheory/QTheory/HypergeometricSeries.lean}
\kw{theorem} sum_qChoose_mul_qPochhammer_eq_qPochhammer_mul \imp{R} [CommRing \var{R}] (\var{a} \var{b} \var{c} \var{q} : \var{R}) (\var{n} : ℕ) :
  ∑ \var{k} ∈ Finset.range (\var{n} + \num{1}),
    qChoose \var{q} \var{n} \var{k} * (\var{a})_\var{k} * (\var{b})_\var{k} * \var{c} ^ \var{k} * (\var{c})_(\var{n} - \var{k}) * (\var{a} * \var{b} * \var{c} * \var{q} ^ \var{k})_(\var{n} - \var{k}) =
  (\var{a} * \var{c})_\var{n} * (\var{b} * \var{c})_\var{n} :=
\end{code}

\begin{proof}
    Applying Lemma \ref{lem:binomial_cauchy} to $u = bq^k$ and $v = ac$, we can write
    \begin{equation}
        \label{eqn:qPS-lemma}
        (abcq^k)_{n-k} = \sum_{r=0}^{n-k} \begin{bmatrix}
            n-k \\ r
        \end{bmatrix}_q (bq^k)_r (ac)^r (ac)_{n-k-r}
    \end{equation}
    and the right hand side becomes
    \begin{align*}
        &\sum_{k=0}^{n} \sum_{r=0}^{n-k} \qbin{n}{k} \qbin{n-k}{r} (a)_k (b)_k c^k (c)_{n-k} (bq^k)_r (ac)^r (ac)_{n-k-r}\\
        &=\sum_{k=0}^{n} \sum_{r=0}^{n-k} \qbin{n}{k+r}\qbin{k+r}{k} (a)_k (b)_{k+r} a^r c^{k+r} (c)_{n-k} (ac)_{n-k-r} \\
        &= \sum_{m=0}^{n} \qbin{n}{m} (b)_m c^m (c)_{n-m} (ac)_{n-m} \sum_{k=0}^{m} \qbin{m}{k} (a)_k a^{m-k} (cq^{n-m})_{m-k}.
    \end{align*}
    Here we rearranged the double sums, where $m = k + r$ for the second equality.
    We also used Lemma \eqref{lem:binomial_mul}, $(b)_k (bq^k)_r = (b)_{k+r}$ and $(c)_{n-k} = (c)_{n-m} (cq^{n-m})_{m-k}$.
    By Lemma \ref{lem:binomial_cauchy} to $u = cq^{n-m}$ and $v = a$, the inner sum becomes
    $$
        \sum_{k=0}^{m} \qbin{m}{k} (a)_k a^{m-k} (cq^{n-m})_{m-k} = (acq^{n-m})_m.
    $$
    Plugging this into the last sum and using $(ac)_{n-m}(acq^{n-m})_m = (ac)_n$ gives
    \begin{align*}
        (ac)_n \sum_{m=0}^{n} \qbin{n}{m} (b)_m c^m (c)_{n-m} = (ac)_n (bc)_n
    \end{align*}
    by Lemma \ref{lem:binomial_cauchy} again.
\end{proof}

\subsection{Infinite $q$-identities}

Theorems in this section \textit{do} involve topology, by using the infinite $q$-Pochhammer symbol. As a result, they should be assumed to be taking place in a complete strongly non-archimedean ring $R$, so that the infinite $q$-Pochhammer symbols exist by theorem \ref{lem:poch-inf-exists}, assuming $q$ is topologically nilpotent.

We introduce two lemmas to manipulate these objects.

\begin{lemma}[Shifting lemma]
\label{lem:shift}
\begin{equation}
(a; q)_\infty (a; q)_m^{-1} = (aq^m; q)_\infty
\end{equation}
\end{lemma}
\begin{proof}
Substitute the definitions.
\end{proof}

In Lean:

\begin{code}{RogersRamanujan/NumberTheory/QTheory/StrongNonarchimedean.lean}
\kw{theorem} qPochhammerInf_shift_eq_bInv_qPochhammer_mul \imp{R} [CommRing \var{R}] [UniformSpace \var{R}]
  [CompleteSpace \var{R}] [StrongNonarchimedeanRing \var{R}] [IsUniformAddGroup \var{R}]
  \imp{a q : R} (\var{hq} : IsTopologicallyNilpotent \var{q}) (\var{ha} : IsUnit (\var{a}; \var{q})_∞) (\var{m} : ℕ) :
  (\var{a} * \var{q} ^ \var{m}; \var{q})_∞ = bInv (\var{a}; \var{q})_\var{m} * (\var{a}; \var{q})_∞ :=
\end{code}
\begin{lemma}[Splitting lemma]
\label{lem:split}
For any natural number $n \in \N$:
\begin{equation}(a; q)_\infty = (a; q^n)_\infty (aq; q^n)_\infty \cdots (aq^{n-1}; q^n)_\infty\end{equation}
\end{lemma}
\begin{proof}
Substitute the definitions.
\end{proof}

In Lean:

\begin{code}{RogersRamanujan/NumberTheory/QTheory/StrongNonarchimedean.lean}
\kw{theorem} qPochhammerInf_eq_prod_range \imp{R} [CommRing \var{R}] [UniformSpace \var{R}]
  [CompleteSpace \var{R}] [StrongNonarchimedeanRing \var{R}] [IsUniformAddGroup \var{R}]
  \imp{a q : R} \imp{m : ℕ} (\var{hm} : \var{m} ≠ \num{0}) (\var{hq} : IsTopologicallyNilpotent \var{q}) :
  (\var{a}; \var{q})_∞ = ∏ \var{j} ∈ Finset.range \var{m}, (\var{a} * \var{q} ^ \var{j}; \var{q} ^ \var{m})_∞ :=
\end{code}

To prove Bailey's lemma, one needs several transformation formulas for $q$-hypergeometric series, which we introduce here.

\begin{definition}[$q$-hypergeometric series]
Let $r, s \ge 0$. The $q$-hypergeometric series ${}_{r}\phi_{s}$ is defined as
\begin{equation}
    {}_{r}\phi_{s}\left(a_0, \dots, a_{r-1}; b_0, \dots, b_{s-1}; q, t\right) := \sum_{n=0}^{\infty} \frac{(a_0)_n \cdots (a_{r-1})_n}{(q)_n (b_0)_n \cdots (b_{s-1})_n} \left((-1)^n q^{\frac{n(n-1)}{2}}\right)^{s+1-r} t^n
\end{equation}
\end{definition}

In Lean:

\begin{code}{RogersRamanujan/NumberTheory/QTheory/HypergeometricSeries.lean}
\kw{def} qHypergeometricInner \imp{R} [CommRing \var{R}]
    (\var{q} : \var{R}) \imp{r s : ℕ} (\var{a} : Fin \var{r} → \var{R}) (\var{b} : Fin \var{s} → \var{R}) (\var{t} : \var{R}) (\var{e} \var{n} : ℕ) :=
  (∏ \var{i}, (\var{a} \var{i})_\var{n}) * bInv (\var{q})_\var{n} * (∏ \var{j}, bInv (\var{b} \var{j})_\var{n}) *
    ((-\num{1}) ^ \var{n} * \var{q} ^ \var{n}.\prop{choose} \num{2}) ^ \var{e} * \var{t} ^ \var{n}

\kw{def} qHypergeometric \imp{R} [CommRing \var{R}] [TopologicalSpace \var{R}]
    (\var{q} : \var{R}) \imp{r s : ℕ} (\var{a} : Fin \var{r} → \var{R}) (\var{b} : Fin \var{s} → \var{R}) (\var{t} : \var{R}) : \var{R} :=
  ∑' \var{n}, qHypergeometricInner \var{q} \var{a} \var{b} \var{t} (\var{s} + \num{1} - \var{r}) \var{n}
\end{code}

In this paper, we will mostly focus on the case when $r = s + 1$.
The most basic identity is the $q$-binomial series identity on ${}_{1}\phi_{0}$. But we first need a lemma:

\begin{lemma}[Infinite $q$-Binomial theorem]
    \label{lem:binom-thm-inf}\label{lem:poch-inf}
    Assume $(a)_\infty$ exists, i.e. $(a)_n$ converges as $n \to \infty$, and $q$ topologically nilpotent. Then:
    \begin{equation}
        (a)_\infty = \sum_{n=0}^\infty (q)_n^{-1} (-a)^n q^{\binom{n}{2}}
    \end{equation}
\end{lemma}

In Lean:

\begin{code}{RogersRamanujan/NumberTheory/QTheory/BinomialTheorem.lean}
\kw{def} qPochhammerInfInner \imp{R} [CommRing \var{R}] (\var{a} \var{q} : \var{R}) (\var{k} : ℕ) : \var{R} :=
  bInv (\var{q}; \var{q})_\var{k} * (\var{a} ^ \var{k} * \var{q} ^ \var{k}.\prop{choose} \num{2})

\kw{theorem} qPochhammerInf_eq_tsum \imp{R} [CommRing \var{R}] [UniformSpace \var{R}]
  [IsUniformAddGroup \var{R}] [StrongNonarchimedeanRing \var{R}] [CompleteSpace \var{R}] [T2Space \var{R}]
  \imp{a q : R} (\var{hq} : IsTopologicallyNilpotent \var{q} := \kw{by} \kw{simp}) :
  (\var{a}; \var{q})_∞ = ∑' \var{k}, qPochhammerInfInner (-\var{a}) \var{q} \var{k} :=
\end{code}
\begin{proof}
    Take limit as $n \to \infty$ of both sides of \ref{lem:binom-thm}, noting that:
    \begin{equation}
        \lim_{n \to \infty} \qbin{n}{k} = (q)_k^{-1}
    \end{equation}
    while being careful with uniform convergence arguments, and then rename $k$ to $n$.
\end{proof}

We can use Lemma \ref{lem:binom-thm-inf} to prove the following identity for ${}_{1}\phi_{0}$.
\begin{theorem}[$q$-Binomial Series Identity]
    \label{thm:q-binomial-series}
    \begin{equation}
    \label{eqn:q-binomial-series}
        {}_{1}\phi_{0}(a;q,t) := \sum_{n=0}^{\infty} \frac{(a)_n}{(q)_n} t^n = \frac{(at)_\infty}{(t)_\infty}
    \end{equation}
\end{theorem}

In Lean:

\begin{code}{RogersRamanujan/NumberTheory/QTheory/BinomialTheorem.lean}
\kw{theorem} qPochhammerInf_div_qPochhammerInf_eq_tsum \imp{R} [CommRing \var{R}] [UniformSpace \var{R}]
  [IsUniformAddGroup \var{R}] [StrongNonarchimedeanRing \var{R}] [CompleteSpace \var{R}] [T2Space \var{R}]
  \imp{a q z : R} (\var{hz} : IsTopologicallyNilpotent \var{z} := \kw{by} \kw{simp})
  (\var{hq} : IsTopologicallyNilpotent \var{q} := \kw{by} \kw{simp}) :
  ((\var{a} * \var{z}); \var{q})_∞ * bInv ((\var{z}; \var{q})_∞) = ∑' \var{n}, (\var{a}; \var{q})_\var{n} * bInv (\var{q}; \var{q})_\var{n} * \var{z} ^ \var{n} :=
\end{code}
\begin{proof}
    The standard proof over $\mathbb{C}$ shows that the middle and the last term of \eqref{eqn:q-binomial-series} satisfy the same functional equation and agree at $t = 0$, to conclude that they are the same by complex analysis (e.g. see \cite[p. 10]{Andrews1986}).
    In our case (for more general ring $R$), we first use the infinite $q$-Binomial theorem (Lemma \ref{lem:binom-thm-inf}) for $at$:
    \begin{equation}
        (at)_\infty = \sum_{n=0}^\infty (q)_n^{-1} (-at)^n q^{\binom{n}{2}}
    \end{equation}
    Then, multiply $\sum_{m=0}^{\infty} (q)_m^{-1} t^m$ to both sides:
    \begin{equation}
        (at)_\infty \sum_{m=0}^{\infty} (q)_m^{-1} t^m = \left( \sum_{k=0}^\infty (q)_k^{-1} (-at)^k q^{\binom{k}{2}} \right) \left( \sum_{m=0}^{\infty} (q)_m^{-1} t^m \right)
    \end{equation}
    Expanding the double-sum, and grouping by $n := k+m$:
    \begin{equation}
        (at)_\infty \sum_{m=0}^{\infty} (q)_m^{-1} t^m = \sum_{n=0}^\infty (q)_n^{-1} t^n \sum_{k=0}^n \qbin{n}{k} (-a)^k q^{\binom{k}{2}}
    \end{equation}
    Then simplifying using the finite $q$-binomial theorem (\ref{lem:binom-thm}):
    \begin{equation}
        (at)_\infty \sum_{m=0}^{\infty} (q)_m^{-1} t^m = \sum_{n=0}^\infty \frac{(a)_n}{(q)_n} t^n
    \end{equation}
    Now substituting $a=1$ gives:
    \begin{equation}
        (t)_\infty \sum_{m=0}^{\infty} (q)_m^{-1} t^m = 1
    \end{equation}
    which gives us:
    \begin{equation}
    \label{eqn:poch-inv}
        (t)_\infty^{-1} = \sum_{m=0}^{\infty} (q)_m^{-1} t^m
    \end{equation}
    So in conclusion, we have
    \begin{equation}
        \frac{(at)_\infty}{(t)_\infty} = \sum_{n=0}^\infty \frac{(a)_n}{(q)_n} t^n
    \end{equation}
    as required.
\end{proof}

We extract Equation \ref{eqn:poch-inv} above as a lemma:
\begin{lemma}[Expansion of inverse infinite $q$-Pochhammer symbol]
    \label{lem:poch-inf-inv}
    If $a$ and $q$ are both topologically nilpotent, then we have
    \begin{equation}
        (a)_\infty^{-1} = \sum_{m=0}^{\infty} (q)_m^{-1} a^m
    \end{equation}
\end{lemma}

In Lean:

\begin{code}{RogersRamanujan/NumberTheory/QTheory/BinomialTheorem.lean}
\kw{theorem} bInv_qPochhammerInf_eq_tsum \imp{R} [CommRing \var{R}] [UniformSpace \var{R}]
  [IsUniformAddGroup \var{R}] [StrongNonarchimedeanRing \var{R}] [CompleteSpace \var{R}] [T2Space \var{R}]
  \imp{a q : R} (\var{ha} : IsTopologicallyNilpotent \var{a} := \kw{by} \kw{simp})
  (\var{hq} : IsTopologicallyNilpotent \var{q} := \kw{by} \kw{simp}) :
  bInv ((\var{a}; \var{q})_∞) = ∑' \var{n}, bInv (\var{q}; \var{q})_\var{n} * \var{a} ^ \var{n} :=
\end{code}

Another important identity is the following Heine's identity for ${}_{2}\phi_{1}$.
\begin{theorem}[Heine]
\label{thm:heine}
    For $x, y, a, t, q \in R$ with $y, t, q$ topologically nilpotent, we have
    \begin{equation}
        {}_{2}\phi_{1}(x,y;ay;q,t) = \frac{(y)_\infty(xt)_\infty}{(ay)_\infty(t)_\infty} {}_{2}\phi_{1}(a, t;xt;q,y). \label{eqn:heine'}
    \end{equation}
    In particular, when $y$ is invertible, taking $z = ay$ gives
    \begin{equation}
        {}_{2}\phi_{1}(x, y; z; q, t) = \frac{(y)_\infty (xt)_\infty}{(z)_\infty (t)_\infty} {}_{2}\phi_{1}\left(\frac{z}{y}, t ; xt; q, y\right).
        \label{eqn:heine}
    \end{equation}
\end{theorem}

In Lean:

\begin{code}{RogersRamanujan/NumberTheory/QTheory/HypergeometricSeries.lean}
\kw{theorem} qHypergeometric_Heine' \imp{R} [CommRing \var{R}] [UniformSpace \var{R}]
  [IsUniformAddGroup \var{R}] [CompleteSpace \var{R}] [StrongNonarchimedeanRing \var{R}] [T2Space \var{R}]
  \imp{x y a q t : R} (\var{hq} : IsTopologicallyNilpotent \var{q}) (\var{ht} : IsTopologicallyNilpotent \var{t})
  (\var{hy} : IsTopologicallyNilpotent \var{y})
  (\var{hxt} : IsUnit (\var{x} * \var{t})_∞) (\var{hay} : IsUnit (\var{a} * \var{y})_∞) :
  ₂φ₁(\var{x}, \var{y}; \var{a} * \var{y}; \var{q}, \var{t}) =
    (\var{y})_∞ * (\var{x} * \var{t})_∞ * bInv (\var{a} * \var{y})_∞ * bInv (\var{t})_∞ * ₂φ₁(\var{a}, \var{t}; \var{x} * \var{t}; \var{q}, \var{y}) :=

\kw{theorem} qHypergeometric_Heine \imp{R} [CommRing \var{R}] [UniformSpace \var{R}]
  [IsUniformAddGroup \var{R}] [CompleteSpace \var{R}] [StrongNonarchimedeanRing \var{R}] [T2Space \var{R}]
  \imp{x y z q t : R} (\var{hq} : IsTopologicallyNilpotent \var{q}) (\var{ht} : IsTopologicallyNilpotent \var{t})
  (\var{hy} : IsTopologicallyNilpotent \var{y})
  (\var{hxt} : IsUnit (\var{x} * \var{t})_∞) (\var{hyu} : IsUnit \var{y}) (\var{hzu} : IsUnit (\var{z})_∞) :
  ₂φ₁(\var{x}, \var{y}; \var{z}; \var{q}, \var{t}) =
    (\var{y})_∞ * (\var{x} * \var{t})_∞ * bInv (\var{z})_∞ * bInv (\var{t})_∞ * ₂φ₁(\var{z} * bInv \var{y}, \var{t}; \var{x} * \var{t}; \var{q}, \var{y}) :=
\end{code}
\begin{proof}
    By Theorem \ref{thm:q-binomial-series}, we can prove that
    \begin{equation}
        \frac{(ay)_\infty}{(y)_\infty} {}_{2}\phi_{1}(x,y;ay;q,t) = \sum_{n, m = 0}^{\infty} \frac{(x)_n(a)_m}{(q)_n(q)_m} t^n y^m q^{nm}
    \end{equation}
    where the right hand side is symmetric under $a \leftrightarrow x$ and $y \leftrightarrow t$, which gives
    \begin{equation}
        \frac{(ay)_\infty}{(y)_\infty} {}_{2}\phi_{1}(x,y;ay;q,t) = \frac{(xt)_\infty}{(t)_\infty} {}_{2}\phi_{1}(a, t;xt;q,y).
    \end{equation}
    Now set $z = ay$ to get \eqref{eqn:heine}.
\end{proof}

Heine's identity can be used to prove the following $q$-analogue of Gauss's identity for ${}_{2}\phi_{1}$.
\begin{theorem}[$q$-Gauss]
    For $x, y, t, q \in R$ with $y, t, q$ topologically nilpotent, we have
    \begin{equation}
        {}_{2}\phi_{1}\left(x, y; xyt; q, t\right) = \frac{\left(xt\right)_\infty\left(yt\right)_\infty}{(xyt)_\infty \left(t\right)_\infty}. \label{eqn:gauss'}
    \end{equation}
    In particular, when $x, y$ are invertible, taking $z = xyt$ gives
    \begin{equation}
        {}_{2}\phi_{1}\left(x, y; z; q, \frac{z}{xy}\right) = \frac{\left(\frac{z}{x}\right)_\infty\left(\frac{z}{y}\right)_\infty}{(z)_\infty \left(\frac{z}{xy}\right)_\infty}. \label{eqn:gauss}
    \end{equation}
\end{theorem}

In Lean:

\begin{code}{RogersRamanujan/NumberTheory/QTheory/HypergeometricSeries.lean}
\kw{theorem} qHypergeometric_Gauss' \imp{R} [CommRing \var{R}] [UniformSpace \var{R}]
  [IsUniformAddGroup \var{R}] [CompleteSpace \var{R}] [StrongNonarchimedeanRing \var{R}] [T2Space \var{R}]
  \imp{x y q t : R} (\var{hq} : IsTopologicallyNilpotent \var{q}) (\var{ht} : IsTopologicallyNilpotent \var{t})
  (\var{hy} : IsTopologicallyNilpotent \var{y})
  (\var{hxt} : IsUnit (\var{x} * \var{t})_∞) (\var{hxyt} : IsUnit (\var{x} * \var{y} * \var{t})_∞) :
  ₂φ₁(\var{x}, \var{y}; \var{x} * \var{y} * \var{t}; \var{q}, \var{t}) = (\var{x} * \var{t})_∞ * (\var{y} * \var{t})_∞ * bInv (\var{x} * \var{y} * \var{t})_∞ * bInv (\var{t})_∞ :=

\kw{theorem} qHypergeometric_Gauss \imp{R} [CommRing \var{R}] [UniformSpace \var{R}]
  [IsUniformAddGroup \var{R}] [CompleteSpace \var{R}] [StrongNonarchimedeanRing \var{R}] [T2Space \var{R}]
  \imp{x y q z : R} (\var{hxu} : IsUnit \var{x}) (\var{hyu} : IsUnit \var{y})
  (\var{hq} : IsTopologicallyNilpotent \var{q})
  (\var{hzxy} : IsTopologicallyNilpotent (\var{z} * bInv \var{x} * bInv \var{y}))
  (\var{hy} : IsTopologicallyNilpotent \var{y})
  (\var{hz} : IsUnit (\var{z})_∞) (\var{hzy} : IsUnit (\var{z} * bInv \var{y})_∞) :
  ₂φ₁(\var{x}, \var{y}; \var{z}; \var{q}, \var{z} * bInv \var{x} * bInv \var{y}) =
    (\var{z} * bInv \var{x})_∞ * (\var{z} * bInv \var{y})_∞ * bInv (\var{z})_∞ * bInv (\var{z} * bInv \var{x} * bInv \var{y})_∞ :=
\end{code}
\begin{proof}
    It follows by applying Theorem \ref{thm:heine} with $t = \frac{z}{xy}$ and using Theorem \ref{thm:q-binomial-series}.
\end{proof}

We also proved the $q$-Pfaff--Saalsch\"utz identity on ${}_{3}\phi_{2}$, which will be used in the proof of Bailey's lemma for $q$-series.
\begin{theorem}[$q$-Pfaff--Saalsch\"utz]
\label{thm:qPS-original}
    \begin{equation}
        {}_{3}\phi_{2}\left(x, y, q^{-n}; z, \frac{xyq^{1-n}}{z}; q, q\right) = \frac{\left(\frac{z}{x}\right)_n \left(\frac{z}{y}\right)_n}{(z)_n \left(\frac{z}{xy}\right)_n}
    \end{equation}
\end{theorem}

In Lean:

\begin{code}{RogersRamanujan/NumberTheory/QTheory/HypergeometricSeries.lean}
\kw{theorem} qHypergeometric_PfaffSaalschutz \imp{R} [CommRing \var{R}] [UniformSpace \var{R}]
  [IsUniformAddGroup \var{R}] [CompleteSpace \var{R}] [StrongNonarchimedeanRing \var{R}] [T2Space \var{R}]
  \imp{q : R} (\var{hq} : IsTopologicallyNilpotent \var{q}) (\var{x} \var{y} \var{z} : \var{R}) (\var{n} : ℕ)
  (\var{hqu} : IsUnit \var{q}) (\var{hxu} : IsUnit \var{x}) (\var{hyu} : IsUnit \var{y}) (\var{hzu} : IsUnit \var{z})
  (\var{hz} : ∀ \var{k}, IsUnit (\var{z})_\var{k})
  (\var{hxyqz} : ∀ \var{k}, IsUnit (\var{x} * \var{y} * \var{q} * (bInv \var{q}) ^ \var{n} * bInv \var{z})_\var{k})
  (\var{hzxy} : IsTopologicallyNilpotent (\var{z} * bInv \var{x} * bInv \var{y})) :
  ₃φ₂(\var{x}, \var{y}, (bInv \var{q}) ^ \var{n}; \var{z}, \var{x} * \var{y} * \var{q} * (bInv \var{q}) ^ \var{n} * bInv \var{z}; \var{q}, \var{q}) =
    (\var{z} * bInv \var{x})_\var{n} * (\var{z} * bInv \var{y})_\var{n} * bInv (\var{z})_\var{n} * bInv (\var{z} * bInv \var{x} * bInv \var{y})_\var{n} :=
\end{code}

The standard proof of Theorem \ref{thm:qPS-original} is obtained by comparing the coefficients of $t^n$ in the following Rogers' identity
\begin{equation}
    {}_{2}\phi_{1}\left(\frac{z}{x}, \frac{z}{y}; z; \frac{xyt}{q}\right) = \sum_{n=0}^{\infty} \frac{\left(\frac{z}{xy}\right)_n\left(\frac{xyt}{z}\right)^n}{(q)_n} {}_{2}\phi_{1}(x,y;z;q,t).
\end{equation}
(e.g. see \cite{Watson1929}).
Instead, we proved a denominator-cleared version of the identity, which does not include inverses.
In particular, the identity is purely algebraic and we do not require any topology on the ring $R$.
\begin{theorem}[$q$-Pfaff--Saalsch\"utz, denominator-cleared version]
\label{thm:qPS-cleared}
    \begin{equation}
        \sum_{k=0}^{n} (x)_k (y)_k (q^{n+1-k})_k \left(\frac{z}{xy}\right)^{k} (q^{k+1})_{n-k} (zq^k)_{n-k} \left(\frac{z}{xy}\right)_{n-k} = (q)_n \left(\frac{z}{x}\right)_n \left(\frac{z}{y}\right)_n
    \end{equation}
\end{theorem}

In Lean:

\begin{code}{RogersRamanujan/NumberTheory/QTheory/HypergeometricSeries.lean}
\kw{theorem} qPfaffSaalschutz_denom_cleared \imp{R} [CommRing \var{R}] \imp{q : R} (\var{x} \var{y} \var{z} : \var{R}) (\var{n} : ℕ)
  (\var{hx} : IsUnit \var{x}) (\var{hy} : IsUnit \var{y}) :
  \kw{let} \var{c} := \var{z} * bInv \var{x} * bInv \var{y}
  ∑ \var{k} ∈ Finset.range (\var{n} + \num{1}),
    (\var{x})_\var{k} * (\var{y})_\var{k} * (\var{q} ^ (\var{n} + \num{1} - \var{k}))_\var{k} *
      \var{c} ^ \var{k} * (\var{q} ^ (\var{k} + \num{1}))_(\var{n} - \var{k}) * (\var{z} * \var{q} ^ \var{k})_(\var{n} - \var{k}) * (\var{c})_(\var{n} - \var{k}) =
  (\var{q})_\var{n} * (\var{z} * bInv \var{x})_\var{n} * (\var{z} * bInv \var{y})_\var{n} :=
\end{code}
\begin{proof}
Applying Lemma \ref{lem:binomial_cauchy_2} to $a = x$, $b = y$, and $c = \frac{z}{xy}$ gives
\begin{equation}
    \sum_{k=0}^{n} \begin{bmatrix} n \\ k \end{bmatrix}_q (x)_k (y)_k \left(\frac{z}{xy}\right)^k \left(\frac{z}{xy}\right)_{n-k} (zq^k)_{n-k} = \left(\frac{z}{x}\right)_n \left(\frac{z}{y}\right)_n.
\end{equation}
Now, multiply both sides by $(q)_n$ and use
\begin{equation}
    (q^{n+1-k})_k (q^{k+1})_{n-k} = (q)_n \begin{bmatrix} n \\ k \end{bmatrix}_q
\end{equation}
to get the desired identity.
\end{proof}

\subsection{Bailey pairs, Bailey's Lemma, and reusable transformation machinery}

We say that two sequences $(\alpha_n, \beta_n)_{n \ge 0}$ form a \emph{Bailey pair with respect to $a$} if the following identity holds for all $n \ge 0$:
\begin{equation}
    \beta_n = \sum_{r=0}^{n} \frac{\alpha_r}{(q)_{n-r} (aq)_{n+r}}
\end{equation}
It is straightforward to formalize the above definition of a Bailey pair:

\begin{code}{RogersRamanujan/NumberTheory/QTheory/Bailey.lean}
\kw{def} Bailey.IsBaileyPair \imp{R} [CommRing \var{R}] (\var{a} \var{q} : \var{R}) (\var{α} \var{β} : ℕ → \var{R}) : Prop :=
  ∀ (\var{n} : ℕ),
    \var{β} \var{n} = ∑ \var{r} ∈ range (\var{n} + \num{1}),
      \var{α} \var{r} * bInv (qPochhammer \var{q} \var{q} (\var{n} - \var{r})) * bInv (qPochhammer (\var{a} * \var{q}) \var{q} (\var{n} + \var{r}))
\end{code}

Bailey's lemma states that one can produce a new Bailey pair from an old Bailey pair. 

\begin{lemma}[Bailey's lemma]
\label{lem:q-bailey}
Let $(\alpha_n, \beta_n)$ be a Bailey pair with respect to $a$, and $\rho_1, \rho_2$ be parameters. Then:
\begin{align}
    \alpha_n' &:= \frac{(\rho_1;q)_n (\rho_2;q)_n }{(\frac{aq}{\rho_1};q)_n (\frac{aq}{\rho_2};q)_n}\left(\frac{aq}{\rho_1\rho_2}\right)^n \alpha_n \label{eqn:bailey-alphaprime} \\
    \beta_n' &= \sum_{j=0}^{n} \frac{(\rho_1;q)_j (\rho_2;q)_j \left(\frac{aq}{\rho_1\rho_2};q\right)_{n-j}}{(\frac{aq}{\rho_1};q)_n (\frac{aq}{\rho_2};q)_n (q;q)_{n-j}} \left(\frac{aq}{\rho_1\rho_2}\right)^j \beta_j \label{eqn:bailey-betaprime}
\end{align}
also form a Bailey pair with respect to $a$, i.e.
\begin{equation}
    \sum_{j=0}^{n} \frac{(\rho_1;q)_j (\rho_2;q)_j \left(\frac{aq}{\rho_1\rho_2};q\right)_{n-j}}{(\frac{aq}{\rho_1};q)_n (\frac{aq}{\rho_2};q)_n (q;q)_{n-j}} \left(\frac{aq}{\rho_1\rho_2}\right)^j \beta_j = \sum_{j=0}^{n} \frac{(\rho_1;q)_j (\rho_2;q)_j}{(\frac{aq}{\rho_1};q)_j (\frac{aq}{\rho_2};q)_j (q)_{n-j}(aq)_{n+j}}\left(\frac{aq}{\rho_1\rho_2}\right)^j \alpha_j
\end{equation}
\end{lemma}

In Lean:

\begin{code}{RogersRamanujan/NumberTheory/QTheory/Bailey.lean}
\kw{theorem} Bailey.qBaileyLemma \imp{R} [CommRing \var{R}] (\var{a} \var{q} : \var{R}) (\var{α} \var{β} \var{α'} \var{β'} : ℕ → \var{R}) 
  (\var{h} : IsBaileyPair \var{a} \var{q} \var{α} \var{β})
  (\var{ρ₁} \var{ρ₂} : \var{R}) (\var{hρ₁} : IsUnit \var{ρ₁}) (\var{hρ₂} : IsUnit \var{ρ₂}) (\var{hq} : IsUnit \var{q})
  (\var{hqpoc} : ∀ \var{k}, IsUnit (\var{q})_\var{k})
  (\var{haqpoc} : ∀ \var{k}, IsUnit (\var{a} * \var{q})_\var{k})
  (\var{haqρ₁} : ∀ \var{k}, IsUnit (\var{a} * \var{q} * bInv \var{ρ₁})_\var{k})
  (\var{haqρ₂} : ∀ \var{k}, IsUnit (\var{a} * \var{q} * bInv \var{ρ₂})_\var{k})
  (\var{hα'} : ∀ \var{n}, \var{α'} \var{n} = (\var{ρ₁})_\var{n} * (\var{ρ₂})_\var{n} * (\var{a} * \var{q} * bInv (\var{ρ₁} * \var{ρ₂})) ^ \var{n} *
    bInv (\var{a} * \var{q} * bInv \var{ρ₁})_\var{n} * bInv (\var{a} * \var{q} * bInv \var{ρ₂})_\var{n} * \var{α} \var{n})
  (\var{hβ'} : ∀ \var{n}, \var{β'} \var{n} = ∑ \var{j} ∈ range (\var{n} + \num{1}),
    (\var{ρ₁})_\var{j} * (\var{ρ₂})_\var{j} * (\var{a} * \var{q} * bInv (\var{ρ₁} * \var{ρ₂}))_(\var{n} - \var{j}) * (\var{a} * \var{q} * bInv (\var{ρ₁} * \var{ρ₂})) ^ \var{j} *
      bInv (\var{a} * \var{q} * bInv \var{ρ₁})_\var{n} * bInv (\var{a} * \var{q} * bInv \var{ρ₂})_\var{n} * bInv ((\var{q})_(\var{n} - \var{j})) * \var{β} \var{j}) :
  IsBaileyPair \var{a} \var{q} \var{α'} \var{β'} :=
\end{code}

The standard proof of Bailey's lemma is based on the following simple lemma on infinite double sums: 
\begin{lemma}
\label{lem:bailey-basic}
If $\alpha, \beta, \gamma, \delta, u, v$ are sequences satisfying
\begin{align}
    \beta_n &= \sum_{r=0}^{n} \alpha_r u_{n-r} v_{n+r} \label{eqn:bailey-basic-1} \\
    \gamma_n &= \sum_{r=n}^{\infty} \delta_r u_{r-n} v_{r+n} \label{eqn:bailey-basic-2}
\end{align}
then we have
\begin{equation}
    \sum_{n=0}^{\infty} \alpha_n \gamma_n = \sum_{n=0}^{\infty} \beta_n \delta_n.
\end{equation}
\end{lemma}

In Lean:

\begin{code}{RogersRamanujan/NumberTheory/QTheory/Bailey.lean}
\kw{theorem} Bailey.bailey_lemma_basic \imp{R} [CommSemiring \var{R}] [TopologicalSpace \var{R}]
  [IsTopologicalSemiring \var{R}] [T3Space \var{R}] (\var{u} \var{v} \var{α} \var{β} \var{γ} \var{δ} : ℕ → \var{R})
  (\var{hγ} : ∀ \var{n}, Summable \kw{fun} \var{r} ↦ \kw{if} \var{n} ≤ \var{r} \kw{then} \var{δ} \var{r} * \var{u} (\var{r} - \var{n}) * \var{v} (\var{r} + \var{n}) \kw{else} \num{0})
  (\var{hswap} : Summable (Function.uncurry \kw{fun} \var{n} \var{r} : ℕ ↦
    \kw{if} \var{n} ≤ \var{r} \kw{then} \var{α} \var{n} * \var{δ} \var{r} * \var{u} (\var{r} - \var{n}) * \var{v} (\var{r} + \var{n}) \kw{else} \num{0}))
  (\var{hαβ} : ∀ \var{n}, \var{β} \var{n} = ∑ \var{r} ∈ range (\var{n} + \num{1}), \var{α} \var{r} * \var{u} (\var{n} - \var{r}) * \var{v} (\var{n} + \var{r}))
  (\var{hγδ} : ∀ \var{n}, \var{γ} \var{n} = ∑' \var{r}, \kw{if} \var{n} ≤ \var{r} \kw{then} \var{δ} \var{r} * \var{u} (\var{r} - \var{n}) * \var{v} (\var{r} + \var{n}) \kw{else} \num{0}) :
  ∑' \var{n}, \var{α} \var{n} * \var{γ} \var{n} = ∑' \var{n}, \var{β} \var{n} * \var{δ} \var{n} :=
\end{code}
Lemma \ref{lem:bailey-basic} can be proved by exchanging the order of the double summation \cite[Theorem 3.1]{Andrews1986}. This can hold in the generality where $R$ is a topological commutative ring which is a $T_3$-space, where the latter condition is required to exchange the order of the double sum (\texttt{Summable.tsum\_comm}).
Then the standard proof of Lemma \ref{lem:q-bailey} follows by taking
\begin{equation}
    u_n = \frac{1}{(q)_n}, \quad v_n = \frac{1}{(aq)_n}, \quad \delta_n = \frac{(\rho_1)_n (\rho_2)_n (q^{-N})_n q^n}{(\rho_1 \rho_2 q^{-N} / a)_n}
\end{equation}
where the corresponding $\gamma_n$ is
\begin{equation}
    \gamma_n = \frac{(aq / \rho_1)_N (aq / \rho_2)_N(-1)^n (\rho_1)_n (\rho_2)_n(q^{-N})_n}{(aq)_N (aq/\rho_1\rho_2)_N (aq/\rho_1)_n (aq/\rho_2)_n(aq^{N+1})_n} \left(\frac{aq}{\rho_1\rho_2}\right)^n q^{nN - \frac{n(n-1)}{2}}
\end{equation}
by Theorem \ref{thm:qPS-original}.
Instead, we used Theorem \ref{thm:qPS-cleared} (\texttt{qPfaffSaalschutz\_denom\_cleared}) to prove Bailey's lemma.
In particular, we only need to assume the invertibility of $q$-Pochhammer symbols, and we do not need any topology on $R$.

The ``limiting'' version of Bailey's lemma is obtained by taking $\rho_1, \rho_2 \to \infty$. In particular, we have
\begin{equation}
    \lim_{a \to \infty} a^{-n} (a)_n = \lim_{n \to \infty} (a^{-1} - 1) (a^{-1} - q) \cdots (a^{-1} - q^{n-1}) = (-1)^{n} q^{\frac{n(n-1)}{2}}
\end{equation}
and
\begin{equation}
    \lim_{a \to 0} (a)_n = 1.
\end{equation}
After taking the limit, \eqref{eqn:bailey-alphaprime} and \eqref{eqn:bailey-betaprime} become
\begin{align}
    \alpha_n' &= \alpha_n \cdot (q^{\frac{n(n-1)}{2}})^2 \cdot a^n q^n = a^n q^{n^2} \alpha_n \\
    \beta_n' &= \sum_{j=0}^{n} \frac{q^{j(j-1)}}{(q)_{n-j}} \cdot a^j q^j\beta_j = \sum_{j=0}^{n} \frac{a^j q^{j^2}}{(q)_{n-j}} \beta_j
\end{align}
and we get the following lemma.
\begin{lemma}[Bailey's Lemma, $\rho_1, \rho_2 \to \infty$]
    \label{lem:bailey-limit-1}
    Let $(\alpha_n, \beta_n)$ be a Bailey pair.
    Define $\alpha_n'$ and $\beta_n'$ as
    \begin{align}
        \alpha_n' = a^n q^{n^2} \alpha_n, \quad \beta_n' = \sum_{j=0}^{n} \frac{a^j q^{j^2}}{(q)_{n-j}} \beta_j
    \end{align}
    Then $(\alpha_n', \beta_n')$ also forms a Bailey pair, i.e.
    \begin{equation}
    \label{eqn:bailey-limit-1}
        \sum_{j=0}^{n} \frac{\beta_j a^j q^{j^2}}{(q)_{n-j}} = \sum_{r=0}^{n} \frac{\alpha_r a^r q^{r^2}}{(q)_{n-r}(aq)_{n+r}} 
    \end{equation}
\end{lemma}

In Lean:

\begin{code}{RogersRamanujan/NumberTheory/QTheory/Bailey.lean}
\kw{theorem} Bailey.qBaileyLemma_limit \imp{R} [CommRing \var{R}] (\var{a} \var{q} : \var{R}) (\var{α} \var{β} \var{α'} \var{β'} : ℕ → \var{R})
  (\var{h} : IsBaileyPair \var{a} \var{q} \var{α} \var{β})
  (\var{hqpoc} : ∀ \var{k}, IsUnit (\var{q})_\var{k}) (\var{haqpoc} : ∀ \var{k}, IsUnit (\var{a} * \var{q})_\var{k})
  (\var{hα'} : ∀ \var{n}, \var{α'} \var{n} = \var{a} ^ \var{n} * \var{q} ^ (\var{n} ^ \num{2}) * \var{α} \var{n})
  (\var{hβ'} : ∀ \var{n}, \var{β'} \var{n} = ∑ \var{j} ∈ range (\var{n} + \num{1}), \var{a} ^ \var{j} * \var{q} ^ (\var{j} ^ \num{2}) * bInv (\var{q})_(\var{n} - \var{j}) * \var{β} \var{j}) :
  IsBaileyPair \var{a} \var{q} \var{α'} \var{β'} :=
\end{code}
If we further take the limit $n \to \infty$ in \eqref{eqn:bailey-limit-1}, we have $(a)_n \to (a)_\infty$ for any $a$, hence
\begin{lemma}[Bailey's Lemma, $\rho_1, \rho_2 \to \infty$, $n \to \infty$]
    \label{lem:bailey-limit-2}
    Let $(\alpha_n, \beta_n)$ be a Bailey pair. Then
    \begin{equation}
        \sum_{j=0}^{\infty} a^j q^{j^2} \beta_j = \frac{1}{(aq)_\infty} \sum_{r=0}^{\infty} a^r q^{r^2} \alpha_r
    \end{equation}
\end{lemma}

In Lean:

\begin{code}{RogersRamanujan/NumberTheory/QTheory/Bailey.lean}
\kw{theorem} Bailey.qBaileyLemma_limit' \imp{R} [CommRing \var{R}] (\var{a} \var{q} : \var{R}) [UniformSpace \var{R}]
  [IsUniformAddGroup \var{R}] [CompleteSpace \var{R}] [T2Space \var{R}] [StrongNonarchimedeanRing \var{R}]
  (\var{α} \var{β} : ℕ → \var{R}) (\var{h} : IsBaileyPair \var{a} \var{q} \var{α} \var{β})
  (\var{hq} : IsTopologicallyNilpotent \var{q}) (\var{haq} : IsTopologicallyNilpotent (\var{a} * \var{q}))
  (\var{hα} : Tendsto (\kw{fun} \var{r} ↦ \var{a} ^ \var{r} * \var{q} ^ (\var{r} ^ \num{2}) * \var{α} \var{r}) atTop (\nhds \num{0})) :
  ∑' \var{j}, \var{a} ^ \var{j} * \var{q} ^ (\var{j} ^ \num{2}) * \var{β} \var{j} = bInv (\var{a} * \var{q})_∞ * ∑' \var{r}, \var{a} ^ \var{r} * \var{q} ^ (\var{r} ^ \num{2}) * \var{α} \var{r} :=
\end{code}

Although the above informal argument is easy to understand, formalizing it is not, since we are working over a general topological ring, not just for power series rings over $\mathbb{C}$.
Instead, we followed slightly different approaches, based on the following auxiliary lemmas. 

\begin{lemma} We have that
\label{lem:binomial-aux-1}
    \begin{equation}
    \label{eqn:binomial-aux-1}
        \sum_{k=0}^{n} z^k q^{k(k-1)} \begin{bmatrix} n \\ k \end{bmatrix}_q (zq^k)_{n-k} = 1
    \end{equation}
\end{lemma}

In Lean:

\begin{code}{RogersRamanujan/NumberTheory/QTheory/BinomialTheorem.lean}
\kw{theorem} qBinomial_qPochhammer_sum \imp{R} [CommRing \var{R}] \imp{q : R} (\var{n} : ℕ) (\var{z} : \var{R}) :
  ∑ \var{k} ∈ range (\var{n} + \num{1}),
    \var{z} ^ \var{k} * \var{q} ^ (\var{k} * (\var{k} - \num{1})) * qChoose \var{q} \var{n} \var{k} * (\var{z} * \var{q} ^ \var{k}; \var{q})_(\var{n} - \var{k}) = \num{1} :=
\end{code}
\begin{proof}
    Let $S(z, n)$ be the left hand side of \eqref{eqn:binomial-aux-1}.
    We can prove that
    \begin{equation}S(z, n+1) = (1 - zq^{n}) S(z, n) + zq^{n} S(zq, n),\end{equation} and applying induction on $n$ concludes the proof.
\end{proof}

\begin{lemma} We have that
\label{lem:binomial-aux-2}
    \begin{equation}
    \label{eqn:binomial-aux-2}
        \sum_{k=0}^{\infty} z^k q^{k(k-1)} \frac{(zq^k)_\infty}{(q)_k} = 1
    \end{equation}
\end{lemma}

In Lean:

\begin{code}{RogersRamanujan/NumberTheory/QTheory/BinomialTheorem.lean}
\kw{theorem} qBinomial_qPochhammer_tsum \imp{R} [CommRing \var{R}] [UniformSpace \var{R}]
  [IsUniformAddGroup \var{R}] [StrongNonarchimedeanRing \var{R}] [CompleteSpace \var{R}] [T2Space \var{R}]
  \imp{q z : R} (\var{hq} : IsTopologicallyNilpotent \var{q}) :
  ∑' \var{k}, \var{z} ^ \var{k} * \var{q} ^ (\var{k} * (\var{k} - \num{1})) * bInv (\var{q}; \var{q})_\var{k} * (\var{z} * \var{q} ^ \var{k}; \var{q})_∞ = \num{1} :=
\end{code}
\begin{proof}
    Define $l_k := z^k q^{k(k-1)} (zq^k)_\infty / (q)_k$. Using Lemma \ref{lem:binomial-aux-1}, we can prove that
    \begin{equation}
        (zq^n)_\infty = \sum_{k=0}^n l_k (q^{n-k+1})_k.
    \end{equation}
    Also, we have
    \begin{equation}
        \lim_{n \to \infty} \sum_{k=0}^{n} l_k ((q^{n-k+1})_k - 1) = 0,
    \end{equation}
    where combining these gives $\sum_{k=0}^{\infty} = 1$, which is \eqref{eqn:binomial-aux-2}.
\end{proof}

To prove Lemma \ref{lem:bailey-limit-1}, plugging in $\beta_n = \sum_{r=0}^{n} \frac{\alpha_r}{(q)_{n-r} (aq)_{n+r}}$ to \eqref{eqn:bailey-limit-1} and rearranging the sum gives
\begin{align*}
    \sum_{j=0}^{n} \sum_{r=0}^{j} \frac{\alpha_r}{(q)_{j-r} (aq)_{j+r}} \frac{a^j q^{j^2}}{(q)_{n-j}}  = \sum_{r=0}^{n} \alpha_r \sum_{j=r}^{n} \frac{a^j q^{j^2}}{(q)_{j-r} (aq)_{j+r} (q)_{n-j}} \stackrel{?}{=} \sum_{r=0}^{n} \frac{\alpha_r a^r q^{r^2}}{(q)_{n-r}(aq)_{n+r}},
\end{align*}
so it is enough to show that
\begin{equation}
    \sum_{j=r}^{n} \frac{a^j q^{j^2}}{(q)_{j-r}(aq)_{j+r}(q)_{n-j}} = \sum_{j=0}^{n-r} \frac{a^{j+r} q^{(j+r)^2}}{(q)_j (aq)_{j+2r}(q)_{n-r-j}} = \frac{a^r q^{r^2}}{(q)_{n-r}(aq)_{n+r}},
\end{equation}
which follows from Lemma \ref{lem:binomial-aux-1} applied to $z = aq^{2r+1}$ and $n \leftarrow n + r$.

One can prove Lemma \ref{lem:bailey-limit-1} (\texttt{qBaileyLemma\_limit}) by rearranging double sum and applying Lemma \ref{lem:binomial-aux-1} for $z = aq^{2r + 1}$.
The proof of Lemma \ref{lem:bailey-limit-2} (\texttt{qBaileyLemma\_limit'}) uses Lemma \ref{lem:bailey-basic}, by applying it to the sequences
\begin{align*}
    u_n = \frac{1}{(q;q)_n}, \quad
    v_n = \frac{1}{(aq;q)_n}, \quad
    \gamma_n = \frac{a^n q^{n^2}}{(aq;q)_{\infty}}, \quad
    \delta_n = a^n q^{n^2}
\end{align*}
In this case, the equation \eqref{eqn:bailey-basic-1} corresponds to $(\alpha_n, \beta_n)$ being a Bailey pair (with respect to $a$), while \eqref{eqn:bailey-basic-2} corresponds to the identity
\begin{equation}
    \label{eqn:bailey_aux}
    \frac{a^n q^{n^2}}{(aq)_n} = \sum_{r=n}^{\infty} \frac{a^r q^{r^2}}{(q)_{r-n} (aq)_{r+n}}
\end{equation}
which follows from Lemma \ref{lem:binomial-aux-2}, applied to $z = aq^{2n+1}$.

\subsection{A strategy for generalizing to non-archimedean rings}

Identities involving $(a; q)_\infty$ where $q \in R$ is topologically nilpotent require $R$ to be a complete strongly non-archimedean ring in order for $(a; q)_\infty$ to exist. But in some applications, everything involved is topologically nilpotent. In that case, we will be able to drop the ``strongly'' condition from the theorems.

This is possible using a technical construction, which we will call a ``partial'' universal property of (multivariate) power series.

\begin{definition}
    Let $R$ be a set with multiplication, zero, and a topology. Then a subset $S \subseteq R$ is called \textbf{(topologically) bounded} if for every neighborhood $U$ of $0_R$, there exists a neighborhood $V$ of $0_R$ such that $S V \subseteq U$. See \cite[Definition 8.3.8]{GR2004}.
\end{definition}

In Lean:

\begin{code}{RogersRamanujan/Topology/Algebra/Nonarchimedean/Bounded.lean}
\kw{structure} Set.TopologicallyBounded
    \imp{R} [Zero \var{R}] [Mul \var{R}] [TopologicalSpace \var{R}] (\var{S} : Set \var{R}) : Prop \kw{where}
  exists_mul_subset \imp{U : Set R} (\var{hU} : \var{U} ∈ \nhds \num{0}) : ∃ \var{V} ∈ \nhds \num{0}, (\var{S} * \var{V} : Set \var{R}) ⊆ \var{U}
\end{code}

Now we state the ``partial universal property'':
\begin{theorem}
    \label{thm:power-series-eval}
    Let $R$ be a ring, and $S$ be a non-archimedean ring, and $(x_i)_{i \in I}$ be a sequence in $S$ that is ``jointly topologically nilpotent'', and $f: R \to S$ be a ring homomorphism with bounded range. Then we obtain a canonical continuous ring homomorphism $\overline{f}: R\ps{x_i \mid i \in I} \to S$ sending $x_i$ to $x_i$ and $r \in R$ to $f(r)$.
    The topology on $R\ps{x_i}$ here is the jointly $x_i$-adic topology.
\end{theorem}

We do not need the full generality of the above theorem since our index sets in our application are always finite, in which case the undefined ``jointly topologically nilpotent'' condition is equivalent to each element being topologically nilpotent. In Lean:

\begin{code}{RogersRamanujan/RingTheory/MvPowerSeries/Evaluation.lean}
\kw{def} MvPowerSeries.eval \imp{R S σ} [Semiring \var{R}] [CommRing \var{S}] [UniformSpace \var{S}]
  [IsUniformAddGroup \var{S}] [NonarchimedeanRing \var{S}] [CompleteSpace \var{S}] [T2Space \var{S}]
  (\var{f} : \var{R} →+* \var{S}) (\var{q} : \var{σ} → \var{S}) : MvPowerSeries \var{σ} \var{R} →+* \var{S} :=
\end{code}

\section{The formalized proof of Jacobi Triple Product}\label{sec:JTP}

Jacobi Triple Product is the first major stress test for this work because its proof forces the formal development to coordinate several forms of the same identity.  On paper, one often passes freely between theta series, products indexed from $0$ or $1$, and normalized variants obtained by replacing $z$ with $-z$ or shifting the exponent of $q$.  In Lean, these normalizations become explicit lemmas about reindexing, multiplication by units, and equality of infinite products or formal series.  This section describes the chosen proof of \Cref{thm:JTP}, the exact formal statement from which it is derived, and the lemmas that make the statement usable in later applications.

\subsection{The mathematical proof selected for formalization}

We first restate the theorem in terms of $q$-Pochhammer symbols to connect to more of the existing infrastructure:
\begin{equation}\label{eqn:jtp-a-q}
    (q)_\infty (a)_\infty (a^{-1}q)_\infty = \sum_{n \in \Z} (-a)^n q^{\frac{n(n-1)}2}
\end{equation}

Then we make the observation that the choice of having $a$ and $a^{-1}q$ here is rather arbitrary:
\begin{enumerate}
    \item We require $a$ to be invertible here, but we don't really use its inverse $a^{-1}$, but rather we just use $a^{-1}q$. (This holds true even for the right hand side!)
    \item In actual applications for example, we might apply this to $(q^5; q^5)_\infty (q^2; q^5)_\infty (q^3; q^5)_\infty$, inside say $\Z\ps{q}$, where if we take the variables literally we end up with $(q, a) := (q^5, q^2)$, but $q^2$ is not invertible!
\end{enumerate} 

In light of these observations we make the following generalization, for $b, c, q \in R$ satisfying $bc=q$:
\begin{equation}
    \label{eqn:jtp-general}
    (q)_\infty (b)_\infty (c)_\infty = \sum_{n \in \Z} e_{-b,-c}(n) q^{\binom{|n|}{2}}
\end{equation}
where:
\begin{equation}
    e_{b,c}(n) := \begin{cases} b^n & n \ge 0 \\ c^{-n} & n \le 0 \end{cases}
\end{equation}

We note further that there is a natural imbalance in the amount of negative signs, and putting the negative signs on the left hand side vs. on the right hand side produces equivalent mathematical statements, but correspond to different Lean statements. Therefore, in the Lean code, we provide both versions (see Theorem \ref{thm:jtp-general}).

\begin{assumption}
We assume that $R$ is a complete strongly non-archimedean ring and $q$ is topologically nilpotent, but we do not put any assumptions on $b$ and $c$.
\end{assumption}

In the following we will explain the proof, as well as how its ``core essence'' is a countable family of power series identities.

\begin{lemma}
    \label{lem:jtp-a-q}
    Assume that $a$ and $q$ are both topologically nilpotent and invertible. Then Equation \ref{eqn:jtp-a-q} holds, i.e.:
    \begin{equation}(q)_\infty (a)_\infty (a^{-1}q)_\infty = \sum_{n \in \Z} (-a)^n q^{\frac{n(n-1)}2}\end{equation}
\end{lemma}

In our formalization we put the negative signs on the left hand side instead:

\begin{code}{RogersRamanujan/NumberTheory/QTheory/JacobiTripleProduct/NilpotentUnit.lean}
\kw{theorem} jacobi_triple_product_units_of_isTopologicallyNilpotent \imp{R} [CommRing \var{R}]
  [UniformSpace \var{R}] [IsUniformAddGroup \var{R}] [CompleteSpace \var{R}] [StrongNonarchimedeanRing \var{R}]
  [T2Space \var{R}]
  \imp{a q : Rˣ} (\var{hq} : IsTopologicallyNilpotent (\var{q} : \var{R}) := \kw{by} \kw{simp})
  (\var{ha} : IsTopologicallyNilpotent (\var{a} : \var{R}) := \kw{by} \kw{simp}) :
  ((\var{q} : \var{R}))_∞ * ((-(\var{a}⁻¹ * \var{q}) : \var{R}))_∞ * ((-\var{a} : \var{R}))_∞ =
  ∑' \var{n} : ℤ, (↑(\var{a} ^ \var{n}) * ↑(\var{q} ^ choose \var{n} \num{2}) : \var{R}) :=
\end{code}
\begin{proof}
    This follows \cite[Chapter 9]{Chan2009}. We first expand $(a^{-1}q)_\infty$ using Lemma \ref{lem:poch-inf}:
    \begin{equation}
        (a^{-1}q)_\infty = \sum_{m=0}^\infty (q)_m^{-1} (-a^{-1}q)^m q^{\binom{m}{2}}
    \end{equation}
    Now multiply by $(q)_\infty$, and transform using the shifting formula (Lemma \ref{lem:shift}):
    \begin{equation}
        (q)_\infty (a^{-1}q)_\infty = \sum_{m=0}^\infty (q^{m+1})_\infty (-a^{-1}q)^m q^{\binom{m}{2}}
    \end{equation}
    Now the key of this proof: since $(q^{m+1})_\infty = 0$ for $m < 0$, we are free to take the sum over $m \in \Z$ instead:
    \begin{equation}
        (q)_\infty (a^{-1}q)_\infty = \sum_{m \in \Z} (q^{m+1})_\infty (-a^{-1}q)^m q^{\binom{m}{2}}
    \end{equation}
    And now once again expand $(q^{m+1})_\infty$ using Lemma \ref{lem:poch-inf}:
    \begin{equation}
        (q)_\infty (a^{-1}q)_\infty = \sum_{m \in \Z} \sum_{k=0}^\infty (q)_k^{-1} (-q^{m+1})^k q^{\binom{k}{2}} (-a^{-1}q)^m q^{\binom{m}{2}}
    \end{equation}
    Grouping terms and simplifying:
    \begin{equation}
        (q)_\infty (a^{-1}q)_\infty = \sum_{m \in \Z} \sum_{k=0}^\infty (q)_k^{-1} (-1)^{k+m} a^{-m} q^{\binom{k+m+1}{2}}
    \end{equation}
    Using the variable transformation $n := -(m + k)$ (eliminating $m$), noting $\binom{-n+1}{2} = \binom{n}{2}$:
    \begin{equation}
        (q)_\infty (a^{-1}q)_\infty = \sum_{n \in \Z} \sum_{k=0}^\infty (q)_k^{-1} (-a)^n a^k q^{\binom{n}{2}}
    \end{equation}
    Factoring out into double sum and then using the expansion of $(a)_\infty^{-1}$ in Lemma \ref{lem:poch-inf-inv}:
    \begin{equation}
        (q)_\infty (a^{-1}q)_\infty = (a)_\infty^{-1} \sum_{n \in \Z} (-a)^n q^{\binom{n}{2}}
    \end{equation}
    which upon rearranging gives us the identity desired, i.e.:
    \begin{equation}(q)_\infty (a)_\infty (a^{-1}q)_\infty = \sum_{n \in \Z} (-a)^n q^{\frac{n(n-1)}2}\end{equation}
\end{proof}

Now we extract the ``core essence'' of this theorem, which is a countable family of power series identities, that will assist us in proving the more general theorem:

\begin{lemma}
    \label{lem:jtp-iden}
    We have a family of identities indexed by $\ell \in \N$:
    \begin{equation}
        \sum_{k=0}^\infty (q)_k^{-1} (q)_{k+\ell}^{-1} q^{k(k+\ell)} = (q)_\infty^{-1}
    \end{equation}
\end{lemma}

In Lean:

\begin{code}{RogersRamanujan/NumberTheory/QTheory/JacobiTripleProduct/PowerSeriesIdentity.lean}
\kw{theorem} tsum_bInv_qPochhammer_mul_bInv_qPochhammer_mul_pow (\var{R}) [CommRing \var{R}] (\var{n} : ℕ) :
  ∑' \var{k}, bInv (X (R := \var{R}); X)_\var{k} * bInv (X; X)_(\var{k} + \var{n}) * X ^ (\var{k} * (\var{k} + \var{n})) = bInv (X; X)_∞ :=
\end{code}
    
\begin{proof}
    Move $(q)_\infty$ to the right hand side from Lemma \ref{lem:jtp-a-q} to obtain (using $a := -a$):
    \begin{equation}
        (-a)_\infty (-a^{-1}q)_\infty = (q)_\infty^{-1} \sum_{n \in \Z} a^n q^{\frac{n(n-1)}2}
    \end{equation}
    Expand the left hand sides using the expansion for infinite $q$-Pochhammer (Lemma \ref{lem:poch-inf}):
    \begin{equation}
        \sum_{k=0}^\infty \sum_{m=0}^\infty (q)_k^{-1} (q)_m^{-1} a^k (a^{-1} q)^m q^{\binom{k}{2} + \binom{m}{2}} = (q)_\infty^{-1} \sum_{n \in \Z} a^n q^{\frac{n(n-1)}2}
    \end{equation}
    Extract the $a^\ell$ coefficient from both sides. For the left hand side, each summand gives $a^{k-m}$, so the solutions to $k-m = \ell$ are indexed by $(k,m) := (j+\ell, j)$ where $j \in \N$ is a free variable.
    \begin{equation}
        \sum_{j=0}^\infty (q)_{j+\ell}^{-1} (q)_j^{-1} q^{j + \binom{j+\ell}{2} + \binom{j}{2}} = (q)_\infty^{-1} q^{\binom{\ell}{2}}
    \end{equation}
    Canceling $q^{\binom{\ell}{2}}$ from both sides gives the identity as desired:
    \begin{equation}
        \sum_{j=0}^\infty (q)_{j+\ell}^{-1} (q)_j^{-1} q^{j(j+\ell)} = (q)_\infty^{-1}
    \end{equation}
\end{proof}

We note further that extracting the $a^{-\ell}$ coefficients does not actually give us more identities. Moreover, all the coefficients here are natural numbers, so they are in fact identities in $\N\ps{q}$ (or less strongly, $\Z\ps{q}$). By Strategy \ref{thm:power-series-eval}, we can in fact generalize these identities to any non-archimedean ring. Finally, since the coefficients are in $\N$, these identities can be interpreted combinatorially.

Then we can reverse all of the steps above to arrive at the more generalized form above:
\begin{theorem}
    \label{thm:jtp-general}
    Equation \ref{eqn:jtp-general} holds. Namely, we have that
    \begin{equation}(q)_\infty (b)_\infty (c)_\infty = \sum_{n \in \Z} e_{-b,-c}(n) q^{\binom{|n|}{2}}\end{equation}
    where
    \begin{equation}e_{b,c}(n) := \begin{cases} b^n & n \ge 0 \\ c^{-n} & n \le 0 \end{cases}\end{equation}
\end{theorem}

In Lean, we have two variants depending on whether the negative sign is:

\begin{code}{RogersRamanujan/NumberTheory/QTheory/JacobiTripleProduct/Basic.lean}
\kw{def} abPow \imp{R} [Pow \var{R} ℕ] (\var{a} \var{b} : \var{R}) (\var{n} : ℤ) : \var{R} := \kw{match} \var{n} \kw{with}
  | (n : ℕ) => \var{a} ^ \var{n}
  | Int.negSucc \var{n} => \var{b} ^ (\var{n} + \num{1})

\kw{theorem} jacobi_triple_product \imp{R} [CommRing \var{R}] [UniformSpace \var{R}]
  [IsUniformAddGroup \var{R}] [CompleteSpace \var{R}] [StrongNonarchimedeanRing \var{R}] [T2Space \var{R}]
  \imp{a b q : R} (\var{hq} : IsTopologicallyNilpotent \var{q}) (\var{hab} : \var{a} * \var{b} = \var{q}) :
  (\var{q}; \var{q})_∞ * (-\var{a}; \var{q})_∞ * (-\var{b}; \var{q})_∞ = ∑' \var{n} : ℤ, abPow \var{a} \var{b} \var{n} * \var{q} ^ \var{n}.\prop{natAbs}.\prop{choose} \num{2} :=

\kw{theorem} jacobi_triple_product' \imp{R} [CommRing \var{R}] [UniformSpace \var{R}]
  [IsUniformAddGroup \var{R}] [CompleteSpace \var{R}] [StrongNonarchimedeanRing \var{R}] [T2Space \var{R}]
  \imp{a b q : R} (\var{hq} : IsTopologicallyNilpotent \var{q}) (\var{hab} : \var{a} * \var{b} = \var{q}) :
  (\var{q}; \var{q})_∞ * (\var{a}; \var{q})_∞ * (\var{b}; \var{q})_∞ = ∑' \var{n} : ℤ, abPow (-\var{a}) (-\var{b}) \var{n} * \var{q} ^ \var{n}.\prop{natAbs}.\prop{choose} \num{2} :=
\end{code}

\section{The formalized proof of the Rogers--Ramanujan identities}\label{sec:RR}

The Rogers--Ramanujan identities are the main demonstration that the preceding definitions form a working theory.  Their proof is not simply a long calculation: it uses a transformation mechanism, keeps track of two parallel $q$-series, and ends by identifying the resulting products with the two residue classes modulo $5$ appearing in \Cref{thm:RR_Identities}.  The formal proof therefore tests whether the library can support the standard mathematical workflow: define the relevant Bailey pair, apply Bailey's Lemma, simplify the resulting sums, invoke the product machinery, and present the final identities as clean formal-power-series equalities.

\subsection{The Bailey-pair input for the two identities}

For the first Rogers--Ramanujan identity, we use the following Bailey pair with respect to $a = 1$:
\begin{align}
    \alpha_{1, n} &= \begin{cases}
        1 & n = 0 \\
        (-1)^n q^{\binom{n}{2}} (1 + q^n) & n > 0
    \end{cases} \\
    \beta_{1, n} &= \delta_{n,0} = \begin{cases} 1 & n = 0 \\ 0 & n > 0 \end{cases}
\end{align}

\begin{theorem}
    \label{thm:RR1-bailey-pair} We have that the sequence
    $(\alpha_{1,n}, \beta_{1,n})$ forms a Bailey pair, where $\alpha_{1,0} = 1$ and
    \begin{equation}
        0 =  \frac{1}{(q)_{n}^2} + \sum_{r=1}^{n} \frac{(-1)^r q^{\binom{r}{2}}(1 + q^r)}{(q)_{n-r}(q)_{n+r}}
    \end{equation}
    for $n > 0$.
\end{theorem}

In Lean:

\begin{code}{RogersRamanujan/NumberTheory/QTheory/RogersRamanujan.lean}
\kw{def} RogersRamanujan.α₁ \imp{R} [Ring \var{R}] (\var{q} : \var{R}) (\var{n} : ℕ) : \var{R} :=
  \kw{if} \var{n} = \num{0} \kw{then} \num{1} \kw{else} (-\num{1}) ^ \var{n} * \var{q} ^ \var{n}.\prop{choose} \num{2} * (\num{1} + \var{q} ^ \var{n})

\kw{def} RogersRamanujan.β \imp{R} [Zero \var{R}] [One \var{R}] (\var{n} : ℕ) : \var{R} := \kw{if} \var{n} = \num{0} \kw{then} \num{1} \kw{else} \num{0}

\kw{theorem} isBaileyPair_α₁_β \imp{R} [CommRing \var{R}] \imp{q : R} (\var{hpu} : ∀ \var{k}, IsUnit (\var{q}; \var{q})_(\var{k})) :
  IsBaileyPair \num{1} \var{q} (α₁ \var{q}) β :=
\end{code}
\begin{proof}
    See \cite[Section 3.5]{Andrews1986}.
    Proof uses $q$-Binomial theorem.
\end{proof}

The following Bailey pair is used for the second Rogers--Ramanujan identity with respect to $a = q$ (here $\beta_{2, n} = \beta_{1,n}$):
\begin{align}
    \alpha_{2,n} &= (-1)^n q^{\binom{n}{2}} [2n+1]_q = \frac{(-1)^n q^{\binom{n}{2}}(1 - q^{2n+1})}{1 - q} \\
    \beta_{2,n} &= \delta_{n,0} = \begin{cases}
        1 & n = 0 \\ 0 & n > 0
    \end{cases}
\end{align}
\begin{theorem} 
    \label{thm:RR2-bailey-pair} We have that the sequence
    $(\alpha_{2,n}, \beta_{2,n})$ forms a Bailey pair, where $\alpha_{2,0} = 1$ and 
    \begin{equation}
        0 = \sum_{r=0}^{n} \frac{(-1)^r q^{\binom{r}{2}}[2r+1]_q}{(q)_{n-r}(q^2)_{n+r}}
    \end{equation}
    for $n > 0$.
\end{theorem}

In Lean:

\begin{code}{RogersRamanujan/NumberTheory/QTheory/RogersRamanujan.lean}
\kw{def} RogersRamanujan.α₂ \imp{R} [CommRing \var{R}] (\var{q} : \var{R}) (\var{n} : ℕ) : \var{R} :=
  (-\num{1}) ^ \var{n} * \var{q} ^ \var{n}.\prop{choose} \num{2} * qInt \var{q} (\num{2} * \var{n} + \num{1})

\kw{theorem} isBaileyPair_α₂_β \imp{R} [CommRing \var{R}] \imp{q : R} (\var{hpu} : ∀ \var{k}, IsUnit (\var{q}; \var{q})_(\var{k})) :
  IsBaileyPair \var{q} \var{q} (α₂ \var{q}) β :=
\end{code}
\begin{proof}
    See \cite[Section 3.5]{Andrews1986}.
\end{proof}


\subsection{Informal proof of the Rogers--Ramanujan identities}

To derive the Rogers--Ramanujan identities, we apply Bailey's lemma \emph{twice} to the above Bailey pairs $(\alpha_{1,n}, \beta_{1,n})$ and $(\alpha_{2,n},\beta_{2,n})$.
More precisely, if we apply Lemma \ref{lem:bailey-limit-1} to $(\alpha_{1,n}, \beta_{1, n})$, we obtain a new Bailey pair
\begin{align}
    \alpha_{1,n}' &= \begin{cases}
        1 & n = 0 \\ q^{n^2} \cdot (-1)^n q^{\binom{n}{2}} (1 + q^n) & n > 0
    \end{cases} \\
    \beta_{1,n}' &= \frac{1}{(q)_n}.
\end{align}

In Lean:

\begin{code}{RogersRamanujan/NumberTheory/QTheory/RogersRamanujan.lean}
\kw{def} RogersRamanujan.αt₁ \imp{R} [Ring \var{R}] (\var{q} : \var{R}) (\var{n} : ℕ) : \var{R} := \var{q} ^ (\var{n} ^ \num{2}) * α₁ \var{q} \var{n}

\kw{def} RogersRamanujan.βt \imp{R} [CommRing \var{R}] (\var{q} : \var{R}) (\var{n} : ℕ) : \var{R} := bInv (\var{q}; \var{q})_\var{n}

\kw{theorem} isBaileyPair_αt₁_βt \imp{R} [CommRing \var{R}] \imp{q : R} (\var{hpu} : ∀ \var{k}, IsUnit (\var{q}; \var{q})_(\var{k})) :
    IsBaileyPair \num{1} \var{q} (αt₁ \var{q}) (βt \var{q}) := 
  qBaileyLemma_limit \num{1} \var{q} (α₁ \var{q}) β (αt₁ \var{q}) (βt \var{q})
    (isBaileyPair_α₁_β \var{hpu}) \var{hpu} (\kw{by} \kw{simpa}) (\kw{by} \kw{simp} [αt₁]) (\kw{by} \kw{simp} [βt, β])
\end{code}
Applying Lemma \ref{lem:bailey-limit-2} to $(\alpha_{1,n}', \beta_{1,n}')$ gives
\begin{align}
    \sum_{j=0}^{\infty} \frac{q^{j^2}}{(q)_j} &= \frac{1}{(q)_\infty} \sum_{r=0}^{\infty} q^{r^2} \alpha_{1,r}' \nonumber \\
    &= \frac{1}{(q)_\infty} + \frac{1}{(q)_\infty} \sum_{r=1}^{\infty} (-1)^r q^{2r^2 + \binom{r}{2}} (1 + q^r) \nonumber \\
    &= \frac{1}{(q)_\infty}\left(1 + \sum_{r=1}^{\infty} (-1)^r q^{\frac{5}{2}r^2 - \frac{r}{2}} + \sum_{r=1}^{\infty} (-1)^r q^{\frac{5}{2}r^2 + \frac{r}{2}}\right) \nonumber \\
    &= \frac{1}{(q)_\infty} \sum_{r \in \mathbb{Z}} (-1)^r q^{\frac{5}{2}r^2 - \frac{r}{2}} \label{eqn:RR1-sum}
\end{align}
Setting $(a, b, q) := (q^2, q^3, q^5)$ in the Jacobi triple product identity gives
\begin{align}
    (q^5;q^5)_\infty (q^2;q^5)_\infty (q^3;q^5)_\infty &= \sum_{n \ge 0} (-q^2)^{n} q^{5\binom{n}{2}} + \sum_{n > 0} (-q^3)^{n} q^{5\binom{n}{2}} \nonumber \\
    &= \sum_{n \ge 0} (-1)^n q^{\frac{5}{2}n^2 - \frac{n}{2}} + \sum_{n > 0} (-1)^n q^{\frac{5}{2}n^2 + \frac{n}{2}} \nonumber \\
    &= \sum_{n \in \mathbb{Z}} (-1)^n q^{\frac{5}{2}n^2 - \frac{n}{2}} \label{eqn:RR1-JTP}
\end{align}
Comparing \eqref{eqn:RR1-sum} and \eqref{eqn:RR1-JTP} gives the first Rogers--Ramanujan identity
\begin{equation}
    G(q) = \sum_{j=0}^{\infty} \frac{q^{j^2}}{(q;q)_j} = \frac{(q^5;q^5)_\infty (q^2;q^5)_\infty (q^3;q^5)_\infty}{(q;q)_\infty} = \frac{1}{(q;q^5)_\infty (q^4;q^5)_\infty}.
\end{equation}
where the last equality follows from the splitting lemma (Lemma \ref{lem:split}):
\begin{equation}
    (q;q)_\infty = (q;q^5)_\infty (q^2;q^5)_\infty (q^3;q^5)_\infty (q^4;q^5)_\infty (q^5;q^5)_\infty. \label{eqn:qpoch_prod5}
\end{equation}

In Lean:

\begin{code}{RogersRamanujan/NumberTheory/QTheory/RogersRamanujan.lean}
\kw{theorem} first_rogers_ramanujan_hasSum_strong \imp{R} [CommRing \var{R}] [UniformSpace \var{R}]
  [IsUniformAddGroup \var{R}] [CompleteSpace \var{R}] [T2Space \var{R}] [StrongNonarchimedeanRing \var{R}]
  \imp{q : R} (\var{hq} : IsTopologicallyNilpotent \var{q}) :
  HasSum (\kw{fun} \var{j} ↦ \var{q} ^ \var{j} ^ \num{2} * bInv (\var{q}; \var{q})_\var{j}) (bInv (\var{q}; \var{q}^\num{5})_∞ * bInv (\var{q}^\num{4}; \var{q}^\num{5})_∞) :=
\end{code}

Similarly, applying Lemma \ref{lem:bailey-limit-1} to $(\alpha_{2,n},\beta_{2,n})$ gives a new Bailey pair (with respect to $a = q$)
\begin{align}
    \alpha_{2,n}' &= (-1)^n q^{n^2 + n + \binom{n}{2}} [2n+1]_q \\
    \beta_{2,n}' &= \frac{1}{(q)_n}.
\end{align}

In Lean:

\begin{code}{RogersRamanujan/NumberTheory/QTheory/RogersRamanujan.lean}
\kw{def} RogersRamanujan.αt₂ \imp{R} [CommRing \var{R}] (\var{q} : \var{R}) (\var{n} : ℕ) : \var{R} :=
  (-\num{1}) ^ \var{n} * \var{q} ^ (\var{n} ^ \num{2} + \var{n} + \var{n}.\prop{choose} \num{2}) * qInt \var{q} (\num{2} * \var{n} + \num{1})

\kw{theorem} isBaileyPair_αt₂_βt \imp{R} [CommRing \var{R}] \imp{q : R} (\var{hpu} : ∀ \var{k}, IsUnit (\var{q}; \var{q})_(\var{k})) :
    IsBaileyPair \var{q} \var{q} (αt₂ \var{q}) (βt \var{q}) :=
  qBaileyLemma_limit \var{q} \var{q} (α₂ \var{q}) β (αt₂ \var{q}) (βt \var{q}) (isBaileyPair_α₂_β \var{hpu}) \var{hpu}
    (\kw{fun} \var{k} ↦ And.right <| \kw{by} \kw{simpa} [qPochhammer_succ] \kw{using} \var{hpu} (\var{k} + \num{1}))
    (\kw{by} \kw{grind} [αt₂, α₂]) (\kw{by} \kw{simp} [β, βt])
\end{code}
Applying Lemma \ref{lem:bailey-limit-2} to $(\alpha_{2,n}', \beta_{2,n}')$ gives
\begin{equation}
    \sum_{j=0}^{\infty} \frac{q^{j^2 + j}}{(q;q)_j} = \frac{1}{(q^2;q)_\infty} \sum_{r=0}^{\infty} q^{r^2 + r} \alpha_{2,r}' = \frac{1 - q}{(q;q)_\infty} \sum_{r=0}^{\infty}(-1)^r q^{\frac{5}{2}r^2 + \frac{3}{2}r} [2r+1]_q
\end{equation}
By $(1 - q) \cdot [2r+1]_q = 1 - q^{2r+1}$, the right hand side becomes
\begin{equation}
\label{eqn:RR2-sum}
    \frac{1}{(q;q)_\infty} \sum_{r =0}^{\infty} (-1)^r q^{\frac{5}{2}r^2 + \frac{3}{2}r} (1 - q^{2r + 1}) = \frac{1}{(q;q)_\infty} \sum_{r \in \mathbb{Z}} (-1)^r q^{\frac{5}{2}r^2 + \frac{3}{2}r}.
\end{equation}
Applying the Jacobi triple product identity to $a = q$ and $b = q^4$ (and again $q \leftarrow q^5$) gives
\begin{equation}
\label{eqn:RR2-JTP}
    (q^5;q^5)_\infty (q;q^5)_\infty (q^4;q^5)_\infty = \sum_{n \in \mathbb{Z}} (-1)^n q^{\frac{5}{2}n^2 + \frac{3}{2}n}
\end{equation}
and comparing with \eqref{eqn:RR2-sum} and applying \eqref{eqn:qpoch_prod5} implies the second Rogers--Ramanujan identity.

In Lean:

\begin{code}{RogersRamanujan/NumberTheory/QTheory/RogersRamanujan.lean}
\kw{theorem} second_rogers_ramanujan_hasSum_strong \imp{R} [CommRing \var{R}] [UniformSpace \var{R}]
  [IsUniformAddGroup \var{R}] [CompleteSpace \var{R}] [T2Space \var{R}] [StrongNonarchimedeanRing \var{R}]
  \imp{q : R} (\var{hq} : IsTopologicallyNilpotent \var{q}) :
  HasSum (\kw{fun} \var{j} ↦ \var{q} ^ (\var{j} * (\var{j} + \num{1})) * bInv (\var{q}; \var{q})_\var{j})
    (bInv (\var{q} ^ \num{2}; \var{q} ^ \num{5})_∞ * bInv (\var{q} ^ \num{3}; \var{q} ^ \num{5})_∞) :=
\end{code}


\subsection{Relaxing the strongly non-archimedean condition}
\label{sec:relax}
Since these identities only involve $q$ which is topologically nilpotent, they can be seen as identities in $\Z\ps{q}$ (or even $\N\ps{q}$), and by Strategy \ref{thm:power-series-eval} they can be generalized to non-archimedean rings, which is expressed by the following in Lean syntax:

\begin{code}{RogersRamanujan/NumberTheory/QTheory/RogersRamanujan.lean}
\kw{theorem} first_rogers_ramanujan \imp{R} [CommRing \var{R}] [UniformSpace \var{R}]
  [IsUniformAddGroup \var{R}] [CompleteSpace \var{R}] [T2Space \var{R}] [NonarchimedeanRing \var{R}]
  \imp{q : R} (\var{hq} : IsTopologicallyNilpotent \var{q}) :
  ∑' \var{j}, \var{q} ^ (\var{j} ^ \num{2}) * bInv (\var{q}; \var{q})_(\var{j}) = bInv (\var{q}; \var{q}^\num{5})_∞ * bInv (\var{q}^\num{4}; \var{q}^\num{5})_∞ :=

\kw{theorem} second_rogers_ramanujan \imp{R} [CommRing \var{R}] [UniformSpace \var{R}]
  [IsUniformAddGroup \var{R}] [CompleteSpace \var{R}] [T2Space \var{R}] [NonarchimedeanRing \var{R}]
  \imp{q : R} (\var{hq} : IsTopologicallyNilpotent \var{q}) :
  ∑' \var{j}, \var{q} ^ (\var{j} * (\var{j} + \num{1})) * bInv (\var{q}; \var{q})_\var{j} = bInv (\var{q} ^ \num{2}; \var{q} ^ \num{5})_∞ * bInv (\var{q} ^ \num{3}; \var{q} ^ \num{5})_∞ :=
\end{code}

\section{Bonus: Two Identities}

As a bonus, we derive the well-celebrated Euler's Pentagonal Number Theorem (Theorem \ref{thm:PNT}) and Jacobi's identity (Theorem \ref{thm:jacobi-triangular}) as corollaries of the Jacobi triple product identity (Theorem \ref{thm:jtp-general}).

\subsection{Proof and Lean verification}
Here we prove and Lean verify two iconic identities.

\begin{theorem}[Euler's Pentagonal Number Theorem]
\label{thm:PNT}
\begin{equation}(q; q)_\infty = \sum_{k \in \Z} (-1)^k q^{\frac{3k^2-k}{2}} = 1 - q - q^2 + q^5 + q^7 + \cdots\end{equation}
\end{theorem}
\begin{proof}
We first apply Lemma \ref{lem:split} to split $(q; q)_\infty$ into $(q; q^3)_\infty (q^2; q^3)_\infty (q^3; q^3)_\infty$.

We then apply the Jacobi triple product identity (Theorem \ref{thm:jtp-general}) to $(b, c, q) := (q, q^2, q^3)$ to obtain:
\begin{equation}(q^3;q^3)_\infty (q;q^3)_\infty (q^2;q^3)_\infty = \sum_{n \in \Z} e_{-q,-q^2}(n) (q^3)^{\binom{|n|}{2}}\end{equation}
where recall that:
\begin{equation}e_{-q,-q^2}(n) :=
\begin{cases}
(-q)^n & n \ge 0 \\
(-q^2)^{-n} & n \le 0
\end{cases}\end{equation}

Then the theorem follows by comparing each term, i.e. for $n \ge 0$ we have:
\begin{equation}(-q)^n (q^3)^{\binom{n}{2}} = (-1)^n \frac{3n^2-n}{2}\end{equation}
and (indexing the negative part by $-n$):
\begin{equation}(-q^2)^n (q^3)^{\binom{n}{2}} = (-1)^n \frac{3n^2+n}{2}\end{equation}
\end{proof}

We once again use Theorem \ref{thm:power-series-eval} to generalize this to non-archimedean rings, which corresponds to the following in Lean syntax:

\begin{code}{RogersRamanujan/NumberTheory/QTheory/Pentagonal.lean}
\kw{theorem} tsum_qPochhammerInf_self \imp{R} [CommRing \var{R}] [UniformSpace \var{R}]
  [IsUniformAddGroup \var{R}] [NonarchimedeanRing \var{R}] [CompleteSpace \var{R}] [T2Space \var{R}]
  \imp{q : R} (\var{hq} : IsTopologicallyNilpotent \var{q} := \kw{by} \kw{simp}) :
  ∑' \var{k} : ℤ, \var{k}.\prop{negOnePow} • \var{q} ^ pentagonal \var{k} = (\var{q}; \var{q})_∞ :=
\end{code}

\begin{theorem}[Jacobi's identity]
\label{thm:jacobi-triangular}
\begin{equation}
\label{eqn:jacobi-triangular}
(q;q)_\infty^3 = \sum_{k \in \N} (-1)^k (2k+1) q^{\frac{k(k+1)}{2}}
\end{equation}
\end{theorem}
\begin{proof}
We would apply the Jacobi triple product identity (Theorem \ref{thm:jtp-general}) to $(1, q, q)$, but the first factor of $(1;q)_\infty$ is $1-1 = 0$, so we would need to cancel that factor formally.

This means we first adjoin a formal invertible element $t$ so that $1 - t$ is regular, i.e. multiplication by $1 - t$ is injective. In the actual Lean code, we consider the ring $R[t,t^{-1}]\ps{q}$ where $t$ is the formal invertible element.

Then applying the Jacobi triple product identity (Theorem \ref{thm:jtp-general}) on $(t, t^{-1}q, q)$ gives us:
\begin{equation}(t;q)_\infty (t^{-1}q;q)_\infty (q; q)_\infty = \sum_{n \in \Z} e_{-t,-t^{-1}q}(n) q^{\binom{|n|}{2}}\end{equation}

The left hand side contains $1-t$ as the first factor of $(t;q)_\infty$, and the right hand side can be grouped ($0$ with $1$, then $-1$ with $2$, then $-2$ with $3$, etc.) to form:
\begin{equation}(1-t)-t^{-1}q(1-t^3)+t^{-2}q^3(1-t^5)+\cdots\end{equation}
so that $1-t$ can also be factored from the right hand side, giving us:
\begin{equation}(tq;q)_\infty (t^{-1}q;q)_\infty (q;q)_\infty = \sum_{k\in\N} (-1)^k t^{-k} (1+t+\cdots+t^{2n}) q^{\binom{n+1}{2}}\end{equation}
where substituting $t=1$ then gives us the desired Equation \ref{eqn:jacobi-triangular}.
\end{proof}

In Lean:

\begin{code}{RogersRamanujan/NumberTheory/QTheory/JacobiTriangular.lean}
\kw{theorem} tsum_qPochhammerInf_self_pow_three \imp{R} [CommRing \var{R}] [UniformSpace \var{R}]
  [IsUniformAddGroup \var{R}] [NonarchimedeanRing \var{R}] [CompleteSpace \var{R}] [T2Space \var{R}]
  \imp{q : R} (\var{hq} : IsTopologicallyNilpotent \var{q} := \kw{by} \kw{simp}) :
  ∑' \var{k} : ℕ, (-\num{1}) ^ \var{k} * (\num{2} * \var{k} + \num{1}) * \var{q} ^ (\var{k} + \num{1}).\prop{choose} \num{2} = (\var{q}; \var{q})_∞ ^ \num{3} :=
\end{code}

\section{Future formalization directions}

We believe that our formalization will be useful for future formalization projects.
For example, our formalization of the Jacobi triple product identity can be used to prove important results on partition numbers, such as the Ramanujan congruence $p(5n+4) \equiv 0 \pmod{5}$.
Also, we already have Dedekind eta function\footnote{\url{https://leanprover-community.github.io/mathlib4_docs/Mathlib/NumberTheory/ModularForms/DedekindEta.html}} and Jacobi theta function\footnote{\url{https://leanprover-community.github.io/mathlib4_docs/Mathlib/NumberTheory/ModularForms/JacobiTheta/OneVariable.html}} in \texttt{mathlib}, and combined with Jacobi Triple Product, we will be able to prove identities between them and also prove sum of two and four squares identities.

Further applications come from Bailey's Lemma. Although we only give one (major) application of it --- the Rogers--Ramanujan identities --- Bailey's lemma can be used to prove a vast number of identities on $q$-series that appear in several contexts (for example, see \cite{Andrews1974, Andrews1986, Bailey1949, Berndt, BFOR, GOW, SillsBook}). Bailey’s lemma is one of the central mechanisms in the theory of $q$-series, providing a systematic way to generate infinite families of identities from a single Bailey pair. Its most classical application is the proof and discovery of Rogers–Ramanujan-type identities as mentioned here. However, the applications also include many generalizations such as the Andrews–Gordon identities. Beyond product-sum identities, Bailey’s lemma has also been used to derive partition identities, transformations of basic hypergeometric series, and identities connected with mock theta functions and modular forms. In this sense, it serves not merely as a proof technique, but as a powerful engine for producing new structure throughout the theory of $q$-series.

At last, since we formalized $q$-theory in (strongly) non-archimedean rings, this will allow us to formalize theory of $p$-adic or mod $p$ modular forms, or modular forms over general rings, in the sense of Katz \cite{katz} or Serre \cite{serre}.
These will be crucial if one wants to formalize general results on congruences of coefficients of modular forms.

\section{Appendix: The AxiomProver system}\label{sec:Appendix}

This appendix describes the AxiomProver system as it was used in the project.  The goal is to make clear which parts of the development are mathematical design choices, which parts are conventional Lean engineering, and which parts were assisted by AxiomProver.  We record the system's role in proposing intermediate lemmas, organizing proof dependencies, and accelerating routine formal steps, while emphasizing that the theorems reported in the paper are certified by the Lean kernel and by the accompanying formal artifact.  The appendix is intended to help readers reproduce the development and to clarify how such systems can support, rather than replace, mathematical judgment.

\subsection{Overview of AxiomProver}
\label{sec:axiomprover}

AxiomProver is an AI system currently under development at Axiom Math. We submitted various problems to AxiomProver throughout the process of making the Rogers--Ramanujan repository, which then autonomously produced formalizations. These problems and formalizations are described in Section \ref{sec:artifacts}.

\subsection{Artifacts}
\label{sec:artifacts}

In each subsection we describe the input and output of AxiomProver. These are hosted in \url{https://github.com/AxiomMath/RogersRamanujan-artifacts}.

Note that these artifacts were produced \textit{during} the process of developing the library, which means that they might no longer compile in the current environment. In that case, we provide updated versions of those files alongside the original ones. These have ``updated'' in their filenames and are not separately listed below.

\subsubsection{Jacobi Triple Product}

We asked AxiomProver to prove the Jacobi triple product identity (Theorem \ref{thm:JTP}). As AxiomProver was under development, we provided fragments of drafts of the RogersRamanujan library.

\begin{itemize}
\item Input:
\begin{itemize}
\item Draft versions of \texttt{DiscreteEval.lean}, \texttt{QTheory.lean}, and \texttt{TruncPoly.lean} (\textit{not provided}).
\item \texttt{chapter9.tex}: A tex reproduction of \cite[Chapter 9]{Chan2009}.
\item \texttt{problem.lean}: A (human) formalized statement of the problem without proof.
\end{itemize}
\item Environment: We used Lean 4.26.0 with mathlib 4.26.0.
\item Output: It then produced \texttt{solution.lean}, a formal proof matching \texttt{problem.lean}.
\end{itemize}

\subsubsection{Limit of topologically nilpotent elements}

We asked AxiomProver to prove that, in a strongly non-archimedean ring, the limit of a convergent sequence of topologically nilpotent elements is again topologically nilpotent. In the actual library we generalized it to the statement that the set of topologically nilpotent elements is closed, but it did not fit the scope of this paper.

\begin{itemize}
\item Input:
\begin{itemize}
\item \texttt{problem.lean}: A formal statement of the question without proof.
\end{itemize}
\item Environment: We used Lean 4.28.0 with mathlib 4.28.0 and the RogersRamanujan library.
\item Output: It then produced \texttt{solution.lean}, a formal proof matching \texttt{problem.lean}.
\end{itemize}

\subsubsection{Finding an open subring}

While we were developing the theory of strongly non-archimedean rings, we asked ourselves if they are just all given by Huber pairs. We then proposed that if so, the open subring would be given by the ``nilpotent-stable'' elements, i.e.
\begin{equation}
R_0 := \{t \in R \mid \forall x \in R, x~\mathrm{topologically~nilpotent} \implies tx~\mathrm{topologically~nilpotent}\}
\end{equation}

We then asked AxiomProver whether $R_0$ is open in general, where we only asked about non-archimedean rings. AxiomProver provided the counter-example $(\Q_2)^\N$, which happened to be strongly non-archimedean anyway. See \ref{sec:strong-not-open} for more discussion.

\begin{itemize}
\item Input:
\begin{itemize}
\item \texttt{problem.lean}: A formal statement of the question without proof, presented in the negative.
\end{itemize}
\item Environment: We used Lean 4.28.0 with mathlib 4.28.0.
\item Output: It then produced \texttt{solution.lean}, a formal proof of the negative, matching \texttt{problem.lean}.
\end{itemize}

\subsubsection{Euler's Pentagonal Number Theorem}

We asked AxiomProver to prove Euler's Pentagonal Number Theorem (Theorem \ref{thm:PNT}).

\begin{itemize}
\item Input:
\begin{itemize}
\item \texttt{task.md}: A short description of the task.
\end{itemize}
\item Environment: We used Lean 4.30.0 with mathlib 4.30.0 and the RogersRamanujan library.
\item Output: It then produced \texttt{problem.lean}, a formal statement of the task, and \texttt{solution.lean}, a formal proof matching \texttt{problem.lean}.
\end{itemize}

\subsubsection{Jacobi's Identity}

We asked AxiomProver to prove Jacobi's identity (Theorem \ref{thm:jacobi-triangular}).

\begin{itemize}
\item Input:
\begin{itemize}
\item \texttt{task.md}: A short description of the task.
\end{itemize}
\item Environment: We used Lean 4.30.0 with mathlib 4.30.0 and the RogersRamanujan library.
\item Output: It then produced \texttt{problem.lean}, a formal statement of the task, and \texttt{solution.lean}, a formal proof matching \texttt{problem.lean}.
\end{itemize}

\subsubsection{Proof of the $q$-Pfaff--Saalsch\"utz identity}

The proof of Lemma \ref{lem:binomial_cauchy_2} and Theorem \ref{thm:qPS-cleared}, the denominator-cleared version of $q$-Pfaff--Saalsch\"utz identity, is autoformalized by AxiomProver.
\begin{itemize}
    \item Input:
    \begin{itemize}
        \item \texttt{problem.lean}: A formal statement of the task
    \end{itemize}
    \item Environment: We used Lean 4.31.0 with mathlib 4.31.0.
    \item Output: It then produced \texttt{solution.lean}, a formal proof matching \texttt{problem.lean}.
\end{itemize}

\section{Appendix: Counterexamples}\label{sec:Counterexamples}

\subsection{Finding an open subring}\label{sec:strong-not-open}

$(\Q_2)^\N$ is a source of counter-examples for many claims regarding strongly non-archimedean rings. See Section \ref{sec:strong} for why it is strongly non-archimedean.

For example, the set of power-bounded elements is $(\Z_2)^\N$, which is not open.

\subsection{A non-archimedean ring that is not strongly non-archimedean}

We produce $R$ with $a, q \in R$ subject to the following constraints: $R$ should be a complete non-archimedean ring, $q$ should be topologically nilpotent, but $(a; q)_\infty$ should not exist.

We start with $\Q_p$ and its usual $p$-adic norm $\|\cdot\|_p$, which allows us to define a family of extended norms (i.e. norms that can take the value $\infty$) on $(\Q_p)^\N$, indexed by $N \in \N^+$:
\begin{equation}\|x\|_N := \sup_{n} \left( p^{-n^3/N} \|x_n\|_p \right)\end{equation}

Now let $A_p := \{x \mid \forall N, \|x\|_N < \infty\}$, so that each $\|\cdot\|_N$ becomes an actual norm. We then let these norms generate a topology on $A_p$, so that for any given sequence $(x_n)_n \in A_p$, we have $x_n \to 0$ as $n \to \infty$ iff for all $N \in \N^+$ we have $\lim_{n \to \infty} \|x_n\|_N = 0$. For $A_p$ we have the following in Lean syntax:

\begin{code}{RogersRamanujan/Counterexamples/NotStrong.lean}
\kw{def} NotStrong.FrechetSpace (\var{p} : ℕ) [Fact \var{p}.\prop{Prime}] : Type :=
  \imp{ \var{x} : ℕ → ℚ_[\var{p}] // ∀ \var{N} : ℕ+, ∃ \var{M} : ℝ, ∀ \var{n}, frechet \var{N} \var{x} \var{n} ≤ \var{M} }
\end{code}

\begin{lemma}
    $A_p$ is a non-archimedean ring under that topology. In particular, multiplication is continuous.
\end{lemma}
\begin{proof}
    The only non-trivial part to check is the continuity of multiplication, which follows from:
    \begin{equation}
        p^{-n^3/N} \|x_n y_n\| = \left( p^{-n^3/(2N)} \|x_n\| \right) \left( p^{-n^3/(2N)} \|y_n\| \right)
    \end{equation}
\end{proof}

In Lean:

\begin{code}{RogersRamanujan/Counterexamples/NotStrong.lean}
\kw{instance} \imp{p : ℕ} [Fact \var{p}.\prop{Prime}] : NonarchimedeanRing (FrechetSpace \var{p})
\end{code}

Now let $a_p := (p^{-n^2})_n$ and $q_p := (p)_n$. We know that $a_p \in A_p$ because $\lim_{n \to \infty} p^{n^2 - \frac{n^3}{N}} = 0$ for any fixed $N \in \N^+$. Similarly we have $q_p \in A_p$. We also know that $q_p$ is topologically nilpotent by an explicit computation.

In Lean:

\begin{code}{RogersRamanujan/Counterexamples/NotStrong.lean}
\kw{theorem} NotStrong.FrechetSpace.isTopologicallyNilpotent_p (\var{p} : ℕ) [Fact \var{p}.\prop{Prime}] :
  IsTopologicallyNilpotent (\var{p} : FrechetSpace \var{p}) :=
\end{code}

However, $\|(a_p)^m (q_p)^{\binom{m}{2}}\|_N = \sup_n \left( p^{-\frac{n^3}{N} + mn^2 - \binom{m}{2}} \right)$ is a simple optimization problem, optimized by $n = \frac{2mN}{3}$ (with marginal errors due to using integers), and putting $N := 1$ gives us:
\begin{equation}
    \|(a_p)^m (q_p)^{\binom{m}{2}}\|_1 \approx \frac{4}{27} m^3 - \binom{m}{2}
\end{equation}
which does not go to $0$ as $m \to \infty$, so $(a_p)^m (q_p)^{\binom{m}{2}}$ does not go to $0$ as $m \to \infty$. In Lean:

\begin{code}{RogersRamanujan/Counterexamples/NotStrong.lean}
\kw{theorem} NotStrong.FrechetSpace.not_tendsto_a_pow_mul_p_pow_choose_two_zero \imp{p : ℕ}
  [Fact \var{p}.\prop{Prime}] : ¬Tendsto (\kw{fun} \var{n} ↦ (a \var{p} ^ \var{n} * \var{p} ^ \var{n}.\prop{choose} \num{2})) atTop (\nhds \num{0}) :=
\end{code}

From this we can already deduce:
\begin{theorem}\label{thm:not-strong}
    $A_p$ is not strongly non-archimedean.
\end{theorem}
\begin{proof}
    Follows from Theorem \ref{thm:strong-prop}.
\end{proof}

In Lean:

\begin{code}{RogersRamanujan/Counterexamples/NotStrong.lean}
\kw{theorem} NotStrong.FrechetSpace.not_strongNonarchimedean (\var{p} : ℕ) [Fact \var{p}.\prop{Prime}] :
  ¬StrongNonarchimedeanRing (FrechetSpace \var{p}) :=
\end{code}

But we can obtain a stronger result:
\begin{theorem}\label{thm:no-poch-inf}
    $(a_p; q_p)_\infty$ does not exist.
\end{theorem}
\begin{proof}
    If it exists, it would be $\lim_{m \to \infty} (a_p; q_p)_m$, and since $A_p$ is non-archimedean, the limit exists if and only if $\lim_{m \to \infty} \left( (a_p; q_p)_m - (a_p; q_p)_{m+1} \right) = 0$, but:
\[
\begin{array}{rcl}
    \| (a_p; q_p)_m - (a_p; q_p)_{m+1} \|_1
    &=& \sup_n \Big( p^{-n^3} \Big( \prod_{i=0}^{m-1} \| 1 - p^{-n^2+i} \|_p \Big) \| p^{-n^2+m} \|_p \Big) \\
    &\ge& p^{-(\sqrt{m})^3 + \sum_{i=0}^{m} ((\sqrt{m})^2-i)} \\
    &=& p^{\frac{m(m+1)}{2} - m^{3/2}}
\end{array}
\]
    which goes to $\infty$ as $m \to \infty$. (For the step involving $\sqrt{m}$, it can be made rigorous by only considering the subsequence where $m$ is a perfect square.)
\end{proof}

In Lean:

\begin{code}{RogersRamanujan/Counterexamples/NotStrong.lean}
\kw{theorem} NotStrong.FrechetSpace.not_tendsto_qPochhammer \imp{p : ℕ} [Fact \var{p}.\prop{Prime}]
  (\var{L} : FrechetSpace \var{p}) : ¬Tendsto ((a \var{p}; \var{p})_·) atTop (\nhds \var{L}) :=
\end{code}


\begin{thebibliography}{99}



\bibitem{Andrews1974}
G.~E. Andrews, \emph{An analytic generalization of the Rogers--Ramanujan identities for odd moduli}, Proc. Natl. Acad. Sci. USA \textbf{71} (1974), no.~10, 4082--4085.

\bibitem{Andrews1975}
G.~E. Andrews, \emph{Problems and prospects for basic hypergeometric functions}, in R.~Askey (ed.), \emph{Theory and Application of Special Functions}, Academic Press, New York, 1975, 191--224.

\bibitem{Andrews1986}
G.~E. Andrews, \emph{q-Series: Their development and application in analysis, number theory, combinatorics, physics, and computer algebra}, American Mathematical Soc., 1986

\bibitem{AndrewsBook}
G.~E. Andrews, \emph{The Theory of Partitions}, Encyclopedia of Mathematics and its Applications, vol.~2, Addison--Wesley, Reading, MA, 1976; reprinted by Cambridge Univ.\ Press, 1998.

\bibitem{AndrewsPartition}
G.~E. Andrews and K. Eriksson. \emph{Integer Partitions}. Cambridge University Press, 2004.

\bibitem{Bailey1949}
W.~N. Bailey, \emph{Identities of the Rogers--Ramanujan type}, Proc. London Math. Soc. (2) \textbf{50} (1949), 1--10.

\bibitem{BaxterBook}
R.~J. Baxter, \emph{Exactly Solved Models in Statistical Mechanics}, Academic Press, London, 1982.

\bibitem{Berndt}
B. C. Berndt. \emph{Ramanujan's Notebooks}. Springer New York, 1991.

\bibitem{BFOR}
K.~Bringmann, A.~Folsom, K.~Ono, and L.~Rolen,
\emph{Harmonic Maass Forms and Mock Modular Forms: Theory and Applications},
American Mathematical Society Colloquium Publications, Vol.~64,
American Mathematical Society, Providence, RI, 2017.

\bibitem{Buzzard2020}
K.~Buzzard, \emph{Division by zero in type theory: a FAQ}, Xena Project blog, 2020. \url{https://xenaproject.wordpress.com/2020/07/05/division-by-zero-in-type-theory-a-faq/}.

\bibitem{Chan2009}
Chan, H. H. (2009). Analytic Number Theory for Undergraduates. In Monographs in number theory. https://doi.org/10.1142/7252


\bibitem{CoddingtonLevinson}
E.~A. Coddington and N.~Levinson, \emph{Theory of Ordinary Differential Equations}, McGraw--Hill, New York, 1955.

\bibitem{Lean}
L.~de~Moura, S.~Kong, J.~Avigad, F.~van~Doorn, and J.~von~Raumer, \emph{The Lean theorem prover (system description)}, in \emph{Automated Deduction -- CADE-25}, Lecture Notes in Computer Science, vol.~9195, Springer, Cham, 2015, 378--388.

\bibitem{DiamondShurman}
F.~Diamond and J.~Shurman, \emph{A First Course in Modular Forms}, Graduate Texts in Mathematics, vol.~228, Springer, New York, 2005.

\bibitem{DMS}
P.~Di Francesco, P.~Mathieu, and D.~S\'en\'echal, \emph{Conformal Field Theory}, Graduate Texts in Contemporary Physics, Springer, New York, 1997.

\bibitem{EichlerZagier}
M.~Eichler and D.~Zagier,
\emph{The Theory of Jacobi Forms},
Progress in Mathematics, Vol.~55,
Birkh\"auser Boston, Boston, MA, 1985.

\bibitem{FLM}
I.~B. Frenkel, J.~Lepowsky, and A.~Meurman, \emph{Vertex Operator Algebras and the Monster}, Pure and Applied Mathematics, Academic Press, Boston, 1988.

\bibitem{GasperRahman}
G.~Gasper and M.~Rahman, \emph{Basic Hypergeometric Series}, 2nd ed., Encyclopedia of Mathematics and its Applications, vol.~96, Cambridge Univ.\ Press, Cambridge, 2004.


\bibitem{repoprover}
F. Gloeckle, A. Rammal, C. Arnal, R. Munos, V. Cabannes, G. Synnaeve, A.Hayat, \emph{Automatic Textbook Formalization}, arXiv preprint arXiv:2604.03071.



\bibitem{GOW}
M.~J. Griffin, K.~Ono, and S.~O. Warnaar, \emph{A framework of Rogers–Ramanujan identities and their arithmetic properties}, Duke Math. J. \textbf{165} (2016), no.~8, 1475--1527.

\bibitem{GR2004}
Gabber, O., and Ramero, L, \emph{Foundations for almost ring theory}, arXiv preprint arXiv:math/0409584 (2004).

\bibitem{grinberg}
D. Grinberg, \emph{An introduction to algebraic combinatorics}, arXiv preprint arXiv:2506.00738 (2025).


\bibitem{isabelle}
T. Nipkow, M. Wenzel, and L. C. Paulson, eds. \emph{Isabelle/HOL: a proof assistant for higher-order logic}. Berlin, Heidelberg: Springer Berlin Heidelberg, 2002.

\bibitem{Jacobi1829}
C.~G.~J. Jacobi, \emph{Fundamenta nova theoriae functionum ellipticarum}, Sumtibus Fratrum Borntraeger, Regiomonti, 1829.

\bibitem{Johnson}
W. P. Johnson. An introduction to q-analysis. Vol. 134. American Mathematical Soc., 2020.

\bibitem{KacIDA}
V.~G. Kac, \emph{Infinite-Dimensional Lie Algebras}, 3rd ed., Cambridge Univ.\ Press, Cambridge, 1990.

\bibitem{LepowskyLi}
J.~Lepowsky and H.~S. Li, \emph{Introduction to Vertex Operator Algebras and Their Representations}, Progress in Mathematics, vol.~227, Birkh\"auser, Boston, 2004.


\bibitem{Mathlib2020}
The \texttt{mathlib} Community, \emph{The Lean mathematical library}, in \emph{Proceedings of the 9th ACM SIGPLAN International Conference on Certified Programs and Proofs (CPP 2020)}, ACM, New York, 2020, 367--381.

\bibitem{isabelle-theta}
M. Eberl, Theta Functions, \emph{Archive of Formal Proofs}, December 2024, \url{https://isa-afp.org/entries/Theta_Functions.html}, Formal proof development.

\bibitem{isabelle-pentagonal}
M. Eberl, The Partition Function and the Pentagonal Number Theorem, \emph{Archive of Formal Proofs}, April 2025, \url{https://isa-afp.org/entries/Pentagonal_Number_Theorem.html}, Formal proof development.

\bibitem{isabelle-rr}
M. Eberl, The Rogers–Ramanujan Identities, \emph{Archive of Formal Proofs}, December 2024, \url{https://isa-afp.org/entries/Rogers_Ramanujan.html}, Formal proof development.

\bibitem{katz}
N. M. Katz, \emph{P-adic properties of modular schemes and modular forms.} Modular Functions of One Variable III: Proceedings International Summer School University of Antwerp, RUCA July 17–August 3, 1972. Berlin, Heidelberg: Springer Berlin Heidelberg, 1973. 69-190.




\bibitem{OnoCBMS}
K.~Ono, \emph{The Web of Modularity: Arithmetic of the Coefficients of Modular Forms and $q$-series}, CBMS Regional Conference Series in Mathematics, vol.~102, Amer. Math. Soc., Providence, RI, 2004.

\bibitem{Rogers1894}
L.~J. Rogers, \emph{Second memoir on the expansion of certain infinite products}, Proc. London Math. Soc. \textbf{25} (1894), 318--343.

\bibitem{RR1919}
L.~J. Rogers and S.~Ramanujan, \emph{Proof of certain identities in combinatory analysis}, Proc. Cambridge Philos. Soc. \textbf{19} (1919), 211--216.

\bibitem{serre}
J.-P. Serre, \emph{Formes modulaires et fonctions z\^{e}ta p-adiques.} Modular Functions of One Variable III: Proceedings International Summer School University of Antwerp, RUCA July 17–August 3, 1972. Berlin, Heidelberg: Springer Berlin Heidelberg, 1973. 191-268.

\bibitem{Schoeneberg}
B.~Schoeneberg, \emph{Elliptic Modular Functions: An Introduction}, Springer, Berlin, 1974.

\bibitem{Schur1917}
I.~Schur, \emph{Ein Beitrag zur additiven Zahlentheorie und zur Theorie der Kettenbr\"uche}, Sitzungsberichte der K\"oniglich Preussischen Akademie der Wissenschaften zu Berlin (1917), 302--321.

\bibitem{SillsBook}
A.~V. Sills, \emph{An Invitation to the Rogers--Ramanujan Identities}, CRC Press, Boca Raton, FL, 2017.

\bibitem{SelsamTypeclass}
D.~Selsam, S.~Ullrich, and L.~de~Moura, \emph{Tabled typeclass resolution}, arXiv preprint arXiv:2001.04301 (2020).



\bibitem{Watson1929}
G.~N. Watson, \emph{A new proof of the Rogers--Ramanujan identities}, J. London Math. Soc. \textbf{4} (1929), 4--9.



\bibitem{WZ}
H. S. Wilf, D. Zeilberger. \emph{An algorithmic proof theory for hypergeometric (ordinary and “q”) multisum/integral identities}, Inventiones mathematicae 108, no. 1 (1992), 575-633.


\bibitem{ZagierRR}
D.~Zagier, \emph{Elliptic Modular Forms and Their Applications}, in \emph{The 1-2-3 of Modular Forms} (J.~H. Bruinier, G.~van~der~Geer, G.~Harder, and D.~Zagier, eds.), Universitext, Springer, Berlin, 2008, 1--103.


\end{thebibliography}
\end{document}